\documentclass{amsart}
\usepackage{graphicx}
\usepackage{dcolumn}
\usepackage{bm}
\usepackage{hyperref}
\usepackage[mathlines]{lineno}

\usepackage[T1]{fontenc}
\usepackage[english]{babel}
\usepackage{amsthm}
\usepackage{tabularx}
\usepackage{float}
\usepackage{amscd}
\usepackage{amsfonts}         
\usepackage{amsmath}
\usepackage{amssymb}
\usepackage{enumerate}
\usepackage{tikz}
\usepackage{eucal}
\usepackage{tikz-cd}
\usepackage{mathrsfs}  
\usepackage{mathdots}  

\newcommand{\Ab}{\mathbb{A}}

\newcommand{\Eb}{\mathbb{E}}
\newcommand{\Fb}{\mathbb{F}}
\newcommand{\Gb}{\mathbb{G}}

\newcommand{\Nb}{\mathbb{N}}

\newcommand{\Pb}{\mathbb{P}}
\newcommand{\Qb}{\mathbb{Q}}

\newcommand{\Vb}{\mathbb{V}}

\newcommand{\Zb}{\mathbb{Z}}

\newcommand{\Cc}{\mathcal{C}}
\newcommand{\Dc}{\mathcal{D}}
\newcommand{\Ec}{\mathcal{E}}

\newcommand{\Hc}{\mathcal{H}}

\newcommand{\Oc}{\mathcal{O}}
\newcommand{\Pc}{\mathcal{P}}

\newcommand{\Ls}{\mathscr{L}}

\newcommand{\Ps}{\mathscr{P}}

\newcommand{\Xfr}{\mathfrak{X}}
\newcommand{\Yfr}{\mathfrak{Y}}

\newcommand{\Brm}{\mathrm{B}}

\newcommand{\PCob}{{\underline\Omega}}

\newcommand{\op}{\mathrm{op}}

\newcommand{\CH}{\mathrm{CH}}

\newcommand{\Spec}{\mathrm{Spec}}
\newcommand{\Gr}{\mathrm{Gr}}

\newcommand{\bl}{\mathrm{Bl}}

\newcommand{\wtil}{\widetilde}

\newcommand{\colim}{\mathrm{colim}}

\newcommand{\Fun}{\mathrm{Fun}}
\newcommand{\hook}{\hookrightarrow}
\newcommand{\cl}{\mathrm{cl}}

\newcommand{\CP}{\mathbb{C}\mathrm{P}}

\newcommand{\Id}{\mathrm{Id}}

\newcommand{\xto}{\xrightarrow}
\newcommand{\dash}{{\text -}}

\newcommand{\MGL}{\mathrm{MGL}}
\newcommand{\tto}{\twoheadrightarrow}

\newcommand{\MS}{\mathrm{MS}}
\newcommand{\SH}{\mathrm{SH}}
\newcommand{\syn}{\mathrm{syn}}

\newcommand{\Sm}{\mathrm{Sm}}

\newcommand{\cdh}{\mathrm{cdh}}

\newcommand{\afp}{\mathrm{afp}}
\newcommand{\fp}{\mathrm{fp}}
\newcommand{\MU}{\mathrm{MU}}
\newcommand{\QProj}{\mathrm{QProj}}
\newcommand{\Ring}{\mathrm{Ring}}
\newcommand{\mot}{\mathrm{mot}}
\newcommand{\dbe}{\mathrm{dbe}}
\newcommand{\sbe}{\mathrm{sbe}}
\newcommand{\Sch}{\mathrm{Sch}}
\newcommand{\Nis}{\mathrm{Nis}}
\newcommand{\Sp}{\mathrm{Sp}}
\newcommand{\Bl}{\mathrm{Bl}}
\newcommand{\cofib}{\mathrm{cofib}}
\newcommand{\nil}{\mathrm{nil}}
\newcommand{\Zar}{\mathrm{Zar}}

\newcommand{\Fin}{\mathrm{Fin}}
\newcommand{\SSeq}{\mathrm{SSeq}}
\newcommand{\lax}{\mathrm{lax}}
\newcommand{\abe}{\mathrm{abe}}

\newtheorem{theo}{Tplottin ubuntuheorem}[section]
\theoremstyle{plain}
\newtheorem{thm}[theo]{Theorem}
\newtheorem{lem}[theo]{Lemma}
\newtheorem{prop}[theo]{Proposition}
\newtheorem{cor}[theo]{Corollary}

\newtheorem*{thm*}{Theorem}
\newtheorem*{lem*}{Lemma}
\newtheorem*{prop*}{Proposition}
\newtheorem*{cor*}{Corollary}
\newtheorem{quest}[theo]{Question}

\theoremstyle{definition}
\newtheorem{defn}[theo]{Definition}
\newtheorem{ex}[theo]{Example}
\newtheorem{cons}[theo]{Construction}
\newtheorem{rem}[theo]{Remark}

\newcommand{\rev}{\textcolor{black}}

\title{Motivic Steenrod problem away from the characteristic}
\author{Toni Annala}
\address{School of Mathematics, Institute for Advanced Study, 1 Einstein Drive, Princeton, NJ 08540, USA}
\email{tannala@ias.edu}
\author{Tobias Shin}
\address{Eckhart Hall, 5734 S University Ave, Chicago, IL 60637, USA}
\email{tobiasshin@uchicago.edu}

\date{\today}

\begin{document}

\maketitle

\begin{abstract}
In topology, the Steenrod problem asks whether every singular homology class is the pushforward of the fundamental class of a closed oriented manifold. Here, we introduce an analogous question in algebraic geometry: is every element on the Chow line of the motivic cohomology of $X$ the pushforward of a fundamental class along a projective derived-lci morphism? If $X$ is a smooth variety over a field of characteristic $p \geq 0$, then a positive answer to this question follows up to $p$-torsion from resolution of singularities by alterations. However, if $X$ is singular, then this is no longer necessarily so: we give examples of motivic cohomology classes of a singular scheme $X$ that are not $p$-torsion and are not expressible as such pushforwards. A consequence of our result is that the Chow ring of a singular variety cannot be expressed as a quotient of its algebraic cobordism ring, as suggested by the first-named-author in his thesis. 
\end{abstract}

\tableofcontents

\section{Introduction}\label{sect:intro}

The Steenrod problem in algebraic topology asks if every class in the singular homology of a topological space $X$ is the pushforward of the fundamental class of an oriented closed manifold mapping continuously to $X$ \cite{eilenberg:1949}. In 1954, as part of his work on cobordism theory, Thom showed that the answer to this question depends on the chosen coefficients \cite{thom:1954}: the answer is positive for $H_*(-;\Qb)$ and $H_*(-; \Fb_2)$, and negative for $H_*(-;\Zb)$ and $H_*(-;\Fb_p)$ for odd primes $p$.

The purpose of this work is to study the cohomological version of the analogous question in algebraic geometry. In algebraic geometry, the role of singular cohomology is played by \emph{motivic cohomology} \cite{beilinson-pairing, beilinson:1987, lichtenbaum:1984}, which was constructed for smooth varieties in the guise of higher Chow groups by Bloch \cite{bloch:1986}. Given their central role in algebraic geometry,  motivic cohomology groups have been studied intensively to this day. In particular, Voevodsky gave an alternative construction for the motivic cohomology groups of smooth varieties and extended them to finite type $k$-schemes in such a way that the resulting theory satisfies cdh-descent \cite{voevodsky:FieldMot}. These, \emph{cdh-local} motivic cohomology groups (and related constructions) have attracted a lot of attention due to their relevance for many fundamental questions in algebraic geometry, such as the structure of the motivic Steenrod algebra away from the characteristic \cite{voevodsky:SteenrodAlg,kelly:2018,HKO}. 

Unfortunately, cdh-local motivic cohomology seems not to be the ``correct'' definition of motivic cohomology for singular schemes. For example, there is an Atiyah--Hirzebruch spectral sequence that relates cdh-local motivic cohomology, not to algebraic $K$-theory, but to the homotopy $K$-theory of Weibel \cite{weibel:1989}. Thus, cdh-local motivic cohomology should be the $\Ab^1$-localization of some more fundamental, non-$\Ab^1$-invariant theory that is closely related to algebraic $K$-theory and coincides with the Bloch's higher Chow groups on smooth varieties.

Recently, independent work by two groups of researchers has led to two proposals for such a theory, at least for equicharacteristic schemes. These constructions approach the problem from very different points-of-view: Elmanto--Morrow's construction uses trace methods \cite{elmanto-morrow}, whereas Kelly--Saito's construction uses crucially the pro-cdh topology developed in the same work \cite{kelly-saito}. The end results of these two approaches were identified in \cite[Corollary~1.11]{kelly-saito}. Henceforward, by a motivic cohomology of an equicharacteristic scheme, we mean the motivic cohomology constructed by Elmanto--Morrow and Kelly--Saito.


Next, we explain our formulation of the algebro-geometric analogue of Steenrod problem. Note that, one of the most striking differences between motivic cohomology and singular cohomology is that the motivic cohomology groups $H^{*,*}_\mot(X;\Zb)$ are bigraded. As the fundamental classes and their pushforwards lie in bidegrees $(2n,n)$ of motivic cohomology (the so called \emph{Chow line}), it only makes sense to ask if all cohomology classes in these degrees admit a description as pushforwards of fundamental classes. 

This raises the question: pushforwards along which class of morphisms $f \colon V \to X$ should we consider? The direct translation from topology would suggest that $V$ should be a smooth proper variety and $f$ an arbitrary map of schemes. However, motivic cohomology does not admit Gysin pushforwards $f_!$ along all such maps. Rather, we want $f$ to be a proper derived-lci (also known as quasi-smooth) morphism. This seems to be the correct relative notion of a closed manifold from the point of view of intersection theory, as it contains smooth proper maps, regular embeddings (e.g. divisors), and is closed under compositions as well as fibre products of derived schemes. For technical reasons, we restrict our attention to projective morphisms. Thus, we formulate the following question:

\begin{quest}[Motivic Steenrod problem over a field]
Let $X$ be a derived $k$-scheme, and let $R$ be a ring. Are the motivic cohomology $R$-modules $H^{2n,n}_\mot(X;R)$ generated by classes of form $f_!(1_V)$, where $f \colon V \to X$ is a projective derived-lci morphism of relative dimension $-n$?
\end{quest}

If the answer to the above question is yes for a particular choice of $X$ and coefficient ring $R$, then we say that $H^{*,*}_\mot(X;R)$ has the \emph{Steenrod property}. If $X$ is smooth, then one can use the well-known identification
\[
H^{2n,n}_\mot(X;\Zb) = \CH^n(X)
\]
of degree $(2n,n)$ motivic cohomology with the \emph{Chow ring} of $X$ \cite{fulton:1998} together with resolution of singularities either by blowups \cite{hironaka:1964a,hironaka:1964b} or by alterations \cite{dejong:1996, temkin:2017} to obtain the following result.

\begin{thm*}[Standard]
If $X$ is a smooth variety over a field $k$ of characteristic exponent $e$, then $H^{*,*}_\mot(X;\Zb[e^{-1}])$ has the Steenrod property.
\end{thm*}

Our main result is that the above does not generalize to singular schemes.

\begin{thm}[Main theorem]\label{thm:main}
Let $k$ be a field of characteristic exponent $e$. Then there exists \rev{a finite type} $k$-scheme $X$ such that $H^{*,*}_\mot(X;\Zb[e^{-1}])$ does not have the Steenrod property.
\end{thm}

\rev{In fact, we show the stronger result that for every prime $l$ that is invertible in $k$, there exists a finite type $k$-scheme $X$ such that $H^{*,*}_\mot(X;\Zb_{(l)})$ does not have the Steenrod property. The restriction to away from the characteristic was mainly for technological reasons: using the motivic Steenrod operations at the characteristic constructed recently in \cite{annala-elmanto}, the authors show that for a field $k$ of positive characteristic $p$, there exists a finite type $k$-scheme such that $H^{*,*}_\mot(X;\Zb_{(p)})$ does not have the Steenrod property.}

\begin{rem}[Proper maps versus projective maps]
In the statement of the motivic Steenrod problem, we have assumed $f$ to be projective rather than just proper. This is for the technical reason that, in Section~\ref{subsect:LongkeGysinProj}, we are able to construct Gysin pushforward maps only along those derived-lci maps that factor as a derived regular embedding into a relative projective space. We do not know how to construct Gysin pushforwards along all proper derived-lci maps. However, if one is able find such a construction, then one obtains a proper version of the motivic Steenrod problem, and Theorem~\ref{thm:main} should immediately generalize for that version.
\end{rem}

\subsection{Relation to Chow rings of singular varieties and algebraic cobordism}

The Chow ring of a smooth variety is one of the most fundamental invariants in algebraic geometry, providing a simple setting for doing intersection theory, and thus the foundations to areas such as enumerative geometry \cite{fulton:1998}. Although its origins are ancient, the Chow ring was put on a rigorous footing in the 70s by Fulton and MacPherson, who provided the first rigorous construction of the intersection product. However, despite substantial efforts \cite{fulton:1975, levine-weibel, annala-cob}, for a long time there was no widely accepted definition of Chow rings of singular varieties. However, following the recent advances in motivic cohomology \cite{elmanto-morrow,kelly-saito}, we may define
\begin{equation}\label{eq:ChowDef}
\CH^*(X) := H^{2*,*}_\mot(X; \Zb)    
\end{equation}
for any equicharacteristic scheme $X$. \rev{Recently, Bouis has constructed motivic cohomology in mixed characteristic \cite{bouis1,bouis2}. Thus, the Chow ring is now defined for all schemes. Nonetheless, we will restrict our attention to the equicharacteristic setting here for the sake of simplicity.}

Similarly, Levine--Morel constructed the \emph{algebraic cobordism rings} $\Omega^*(X)$ for smooth varieties $X$ over a field of characteristic 0 as the ring of cobordism classes of smooth varieties that are projective over $X$ \cite{levine-morel}. However, as they made a liberal use of resolution of singularities \cite{hironaka:1964a,hironaka:1964b} and weak factorization \cite{abramovich:2002}, their methods did not generalize further to singular schemes or to positive characteristic. By using derived algebraic geometry, the first-named-author was able to overcome these limitations and define the algebraic cobordism ring $\Omega^*(X)$ as the ring of cobordism classes of projective derived-lci $X$-schemes, and proved that these rings satisfy many good properties (such as projective bundle formula) for all noetherian schemes $X$ that admit an ample line bundle \cite{annala-cob, annala-yokura, annala-chern, annala-thesis}. Moreover, Annala raised the following question \cite{annala-cob,annala-thesis}.

\begin{quest}[Chow ring from cobordism]\label{quest:ChowFromCob}
Is the ring $\Zb_a \otimes_L \Omega^*(X)$ the correct model for the Chow ring of a noetherian derived scheme $X$ that admits an ample line bundle?
\end{quest}
\noindent Above, $L$ is Lazard's coefficient ring of the universal formal group law \cite{lazard:1955}, and $\Zb_a$ is the integers, considered as an $L$-algebra via the additive formal group law $x+y$. Slightly more generally, Adeel Khan has asked if there exists a quotient of the algebraic cobordism ring $\Omega^*(X)$ that gives the correct model of Chow ring (see \cite[\S 6.5.2]{khan:2020}).

Question~\ref{quest:ChowFromCob} was not raised without evidence. Namely: 

\begin{enumerate}    
\item If $X$ is a smooth quasi-projective variety in characteristic 0, then Levine--Morel proved that $\Zb_a \otimes_L \Omega^*(X) \cong \CH^*(X)$ \cite{levine-morel}. See \cite{lowrey--schurg, annala-pre-and-cob} for comparison of Levine--Morel algebraic cobordism with alternative approaches using derived geometry. \label{item:SmCase}

\item An analogous claim for algebraic $K$-theory holds: $\Zb_m \otimes_L \Omega^*(X) \cong K^0(X)$, where $\Zb_m$ is the integers considered as an $L$-module via the multiplicative formal group law $x + y - xy$, and where $K^0(X)$ is the Grothendieck ring of vector bundles on $X$ \cite[Theorem~5.4]{annala-chern}. \label{item:KCase} 

\item The claim holds with rational coefficients: $\Qb_a \otimes_L \Omega^*(X) \cong \Qb \otimes \CH^*(X)$. \rev{In particular, $H^{*,*}(X;\Qb)$ has the Steenrod property for all equicharacteristic schemes $X$ whose Krull dimension is finite and which admit an ample line bundle.\footnote{\rev{In fact, by the results of \cite{annala-base-ind-cob} we can deduce that $\Qb$-coefficient motivic cohomology safisfies the Steenrod property for all $X$ that have a finite Krull dimension and admit an ample family of line bundles.}} } This is because both sides are isomorphic to $\Qb \otimes K^0(X)$ via the Chern character. For the left hand side, this is  \cite[Theorem~5.9]{annala-chern}. For the right hand side, one can prove this by combining the identification of motivic cohomology with the $\gamma$-graded rational $K$-theory \cite[Theorem 1.1(2)]{elmanto-morrow} with the Riemann--Roch theorem for $\gamma$-graded rational $K$-theory \cite[Expose~VIII Theorem~3.6]{berthelot:1971}, and checking that Chern character gives an isomorphism from $\Qb \otimes K_0$ to $\gamma$-graded $\Qb \otimes K_0$. \label{item:QCase}

\item Levine and Weibel have constructed a candidate for the degree $d$ part of the Chow ring of a dimension $d$ equidimensional $k$-scheme $X$ \cite{levine-weibel}. By construction, it is generated by fundamental classes of fields mapping into the non-singular locus of $X$. As such maps are proper and derived-lci, the degree $d$ algebraic cobordism should at least surject to the group of Levine--Weibel zero cycles.

\item The rings $\Zb_a \otimes_L \Omega^*(X)$ form a reasonable cohomology theory, satisfying e.g. the projective bundle formula \cite{annala-yokura, annala-chern}. 
\end{enumerate}

Nonetheless, the answers to the Question~\ref{quest:ChowFromCob} and the more general question of Adeel Khan mentioned above are negative, as not all schemes have the Steenrod property (Theorem~\ref{thm:main}). To understand how this holds up to the evidence above, it is useful to consider the analogous claim in topology. The analogues of Items~(\ref{item:KCase}) and~(\ref{item:QCase}) are true in topology, and follow from the fact that the $L$-algebras $\Zb_m$ and $\Qb_a$ satisfy the assumptions of the Landweber exact functor theorem \cite{landweber:1976}. As the Landweber exact functor theorem is true in motivic homotopy theory (with and without $\Ab^1$-homotopy invariance, see \cite{naumann:2009,AHI:atiyah}), the algebro-geometric facts (\ref{item:KCase}) and (\ref{item:QCase}) do not seem surprising.  

However, the analogue of Item~\ref{item:SmCase} is false in topology. Indeed, the natural map $\mathrm{MU}^*(X) \to H^*(X;\Zb)$ from complex cobordism to singular cohomology need not be a surjection even when $X$ is a manifold. Thus, $\Zb_a \otimes_L \mathrm{MU}^*(X) \to H^*(X;\Zb)$ is not an isomorphism in general.

Under closer inspection, one realizes that Item~(\ref{item:SmCase}) is analogous to the fact that the natural map $\Zb_a \otimes_L \mathrm{MU}^*(X) \to H^*(X;\Zb)$ is an isomorphism in degree 0. To see this, one needs three ingredients:
\begin{enumerate}
    \item \textbf{Hopkins--Morel isomorphism:} There exists a natural map of motivic spectra
    $$\MGL/(a_1,a_2,\cdots) \to H\Zb$$
    that realizes motivic cohomology spectrum $H\Zb$ as an explicit quotient\footnote{If working over a field of positive characteristic $p$, then the map $\MGL/(a_1,a_2,\cdots) \to H\Zb$ is known to be an isomorphism only after inverting $p$. However, it is expected that inverting $p$ is not necessary.} of the algebraic cobordism spectrum $\MGL$ \cite{hoyois:2013}. Above, $a_i \in \MGL^{-2i,-i}(k)$ are the images of polynomial generators of the Lazard ring in the algebraic cobordism of the base field $k$.  
    
    \item \textbf{Levine's comparison theorem:} If $X$ is a smooth quasi-projective variety over a field of characteristic 0, then
    $$\MGL^{2i,i}(X) \cong \Omega^i(X)$$
    \cite{levine:2009}. In other words, the bigraded cohomology theory $\MGL^{i,j}(X)$ is \emph{higher algebraic cobordism}. 
    
    \item \textbf{Vanishing of negative higher cobordism groups:} If $X$ is a smooth variety, the groups $\MGL^{i,j}(X)$ vanish whenever $i > 2j$.\footnote{This is a consequence of the Hopkins--Morel isomorphism, so over a field of positive characteristic $p$, this is known only after inverting $p$ in the coefficients.} 
\end{enumerate}
Indeed, combining the above facts with the fact that, for an element $a \in E^{-2n,-n}(k)$, where $E$ is a motivic ring spectrum,\footnote{More generally, $E$ could be a module over a motivic ring spectrum $R$, and $a \in R^{-2n,-n}(k)$.} one obtains natural long exact quotient sequences
\begin{equation*}
    \cdots \to E^{i+2n,j+n}(X) \xto{a \cdot } E^{i,j}(X) \to (E/a)^{i,j}(X) \to E^{i+2n+1,j+n}(X) \to \cdots
\end{equation*}
one recovers the isomorphism $\Zb_a \otimes_L \Omega^*(X) \cong \CH^*(X)$ for a characteristic 0 smooth variety $X$.

For singular varieties, the negative higher cobordism groups do not always vanish. Ultimately, this is the reason why the rings $\Zb_a \otimes_L \Omega^*(X)$ do not give the right model of Chow rings for singular varieties.

\subsection{Structure of the argument}

Let $k$ be a field of characteristic $p \geq 0$, and $l$ a prime number invertible in $k$. Our aim is to find a singular variety $X$, and an element $a \in \CH^n(X)$ that is not expressible as a linear combination of cycle classes of projective derived-lci $X$-schemes. To do so, we observe that if $a$ were such a cycle class, then its image $b \in H^{2n,n}_\cdh(X;\Zb)$ in the cdh-local motivic cohomology would lift to a cdh-local algebraic cobordism class $\tilde b \in \MGL^{2n,n}_\cdh(X)$. As there are well-known obstructions to liftability of cohomology classes to $\MGL_\cdh$ coming from the motivic Steenrod algebra \cite{voevodsky:2003a}, finding an unliftable class $b$ is not difficult. Moreover, as the difference between $H_\mot$ and $H_\cdh$ is either $p$-complete (if $p>0$) or a $\Qb$ vector space (if $p = 0$) \cite[Theorem 1.5]{elmanto-morrow}, it suffices to find an unliftable torsion element $b \in H^{2n,n}_\cdh(X,\Zb)$ whose order is invertible in $k$, as such a class can always be lifted to an element of $H^{2n,n}_\mot(X,\Zb)$.

We then reduce this to a problem about motivic cohomology of smooth varieties (note that motivic cohomology and the cdh-local motivic cohomology groups coincide for smooth $k$-varieties \cite[Theorem 1.1]{elmanto-morrow}). Starting from a smooth quasi-projective variety $X$, we construct explicit singular affine varieties $X_n$ that are cdh-locally $\Ab^1$-homotopy equivalent to simplicial suspensions of $X$, so that
\begin{equation*}
    \tilde E^{i,j}(X; \Zb) \cong \tilde E^{i+n,j}(X_n; \Zb)
\end{equation*}
for any cdh-local $\Ab^1$-invariant cohomology theory $E$. Thus, it suffices to find a smooth quasi-projective variety $X$, and an unliftable torsion element $b \in H^{i,j}_\mot(X,\Zb)$ in \emph{any} bidegree, whose order is invertible in $k$.

Our argument then proceeds as follows:
\begin{enumerate}        
    \item Liftability of classes along $\vartheta_l: \MGL_\cdh^{i,j}(X) \to H^{i,j}_\cdh(X;\Zb/l)$ is well understood: non-vanishing of $Q_m(c)$, where $Q_m$ is the $m$th \textit{motivic Milnor operation}, is an obstruction for $c$ lifting along $\vartheta_l$ \cite{hoyois:2013}.

    \item If $c \in H^{i,j}_\mot(X; \Zb/l)$ is such that $Q_m\beta(c) \not = 0$, then $b=\tilde\beta(c) \in H^{i+1,j}_\mot(X; \Zb)$ is an $l$-torsion class that lifts $\beta(c)$, and therefore cannot be lifted to $\MGL^{i+1,j}(X)$ either. Here, $\beta$ and $\tilde\beta$ are the \textit{Bockstein operations}
    \begin{align*}
        \beta \colon H^{i,j}_\mot(X; \Zb/l) &\to H^{i+1,j}_\mot(X; \Zb/l) \\
        \tilde\beta \colon H^{i,j}_\mot(X; \Zb/l) &\to H^{i+1,j}_\mot(X; \Zb).
    \end{align*}
    Thus, we are reduced to finding a mod-$l$ motivic cohomology class $c$ on which the operation $Q_m\beta$ does not vanish.

    \item Finally, a class $c$ on which $Q_m\beta$ does not vanish can be found with relative ease in the mod-$l$ motivic cohomology of an approximation of the classifying stack $\Brm \mu_l \times \Brm \mu_l$ by a smooth quasi-projective variety $X$.
\end{enumerate}
The lift $\tilde a$ of this class to the motivic cohomology group $H^{2n,n}_\mot(X_k;\Zb) = \CH^n(X_k)$ of the singular variety $X_k$, where $k = 2j-i-1$, is not an integral combination of cycle classes of a projective derived-lci morphisms, thus proving Theorem~\ref{thm:main}.

\subsection{Conventions}

Throughout the article, we will freely use the language of $\infty$-categories and derived schemes. The standard reference for the theory of $\infty$-categories is \cite{HTT}. The standard references for the theory of derived schemes are \cite{lurie-thesis,SAG, HAG1, HAG2, GR}. Moreover, \cite[Section~2]{annala-base-ind-cob} and \cite[Section~3]{annala-thesis} provide brief introductions to derived algebraic geometry from a point of view that should be useful for the purposes of understanding this article.

We will denote motivic cohomology and its cdh-local variant by $H_\mot^{*,*}$ and $H_\cdh^{*,*}$, respectively. These two theories coincide on smooth varieties over a field. For the purpose of avoiding confusion, we aim to use the notation $H_\mot^{*,*}$ whenever only smooth varieties are considered. The algebraic cobordism spectrum on smooth $k$-varieties and its $\cdh$-local $\Ab^1$-invariant extension to finite type (derived) $k$-schemes are denoted by $\MGL^{*,*}$ and $\MGL^{*,*}_\cdh$, respectively. In particular, we will avoid writing $\MGL^{*,*}(X)$ for a singular variety, unless it is necessary for speculation. This is because there is no generally accepted definition of algebraic cobordism spectrum that satisfies enough of the expected properties to be deserving of the name.

We will denote by $\PCob^*$ and $\Omega^*$ the \textit{geometrically defined} precobordism and algebraic cobordism rings of quasi-projective derived $k$-schemes, as defined in e.g. \cite{annala-thesis,annala-base-ind-cob}. Conjecturally, $\Omega^*$ is to algebraic cobordism spectrum what the Chow ring $\CH^*$ is to motivic cohomology: an explicit description of the Chow line of the bigraded theory. In particular, for all derived schemes $X$, there should be natural identifications $\Omega^n(X) \cong \MGL^{2n,n}(X)$ that give rise to isomorphisms of graded rings. Such an identification was obtained by Marc Levine for smooth quasi-projective varieties in characteristic 0 \cite{levine:2009}. The main obstruction of extending such a result to smooth varieties in positive characteristic is the lack of localization exact sequence for $\Omega^*$ in this generality. Recently, a part of this sequence has been realized by the first-named-author after inverting the characteristic in the coefficients \cite[Theorem~4.19]{annala-spivak}.

\subsection*{Acknowledgements} We would like to thank Longke Tang for sharing his results about Gysin maps in non-$\Ab^1$-invariant motivic homotopy theory. Moreover, we would like to thank Elden Elmanto for answering numerous questions about Elmanto--Morrow's construction of motivic cohomology, and in particular for pointing out that the projective bundle formula for derived schemes can be deduced from \cite[Construction~4.38]{elmanto-morrow} (see Remark~\ref{rem:DerivedPBF}). We would also like to thank Shane Kelly and Adeel Khan for comments on the draft of this article, which helped to improve the introduction, and Marc Hoyois for pointing out a subtlety about the cdh-topology on derived schemes. \rev{Finally, we would like to thank the anonymous referee for careful comments that helped to improve the exposition of the paper.} This paper was written while the first-named-author was in residence at the Institute for Advanced Study in Princeton, and supported by the National Science Foundation (Grant No. DMS-1926686).

\section{Background and preliminary results}

Here, we recall necessary background and preliminary results. In Section~\ref{subsect:EMCoh}, we recall the basic properties of Elmanto--Morrow's construction of motivic cohomology. In Section~\ref{subsect:LongkeGysin}, we recall the basic properties of the Gysin pushforwards along derived regular embeddings constructed by Longke Tang. In Section~\ref{subsect:LongkeGysinProj}, we define Gysin pushforwards along projective lci morphisms of derived schemes. In Section~\ref{subsect:PCob}\rev{,} we recall the definition and the basic properties of the universal precobordism rings $\PCob^*(X)$. Finally, in Section~\ref{subsect:CycleMap}, we use the theory of Gysin pushforwards constructed in Section~\ref{subsect:LongkeGysinProj} to construct \emph{cycle maps} from $\PCob^*$ to sufficiently well-behaved cohomology theories that satisfy excision for derived blowups.

Throughout this section, unless otherwise mentioned, $S$ is an arbitrary base derived scheme, and $k$ is a field. Throughout the section, all schemes are implicitly assumed to be defined over the base $S$, and all products mean fiber products over $S$. 

\subsection{Motivic cohomology of equicharacteristic schemes}\label{subsect:EMCoh}

In \cite{elmanto-morrow}, Elmanto and Morrow construct motivic complexes $\Zb(i)^\mot$ of derived $k$-schemes, therefore also defining the motivic cohomology groups of derived $k$-schemes 
\[
H^{i,j}_\mot(X;\Zb) := H^{i}\big(X; \Zb(j)^\mot\big).
\]
Their motivic complexes arise as graded pieces of a {motivic filtration} on algebraic $K$-theory, and as a result, there is a motivic Atiyah--Hirzebruch spectral sequence that realizes the expected relationship between motivic cohomology groups and the higher \rev{(both positive and negative)} algebraic $K$-theory groups. This, together with other good properties of Elmanto--Morrow's construction \cite[Theorem~1.1]{elmanto-morrow}, provides extremely strong evidence that they have indeed constructed the correct motivic cohomology theory. On smooth $k$-schemes, the motivic complexes agree with those defined by Bloch \cite{bloch:1986}. Kelly--Saito \cite{kelly-saito} offer an alternative construction of the motivic complexes\rev{, at least for noetherian $k$-schemes, as a pro-cdh-local left Kan extension from smooth $k$-schemes of Bloch's cycle complexes.}

Here, we recall the basic properties of the motivic complexes that are necessary for the purposes of this paper. We start with recalling the relationship between the motivic complexes and their $\Ab^1$-invariant counterpart (see \cite[Theorem~1.5]{elmanto-morrow}). This is needed later, when we want to lift classes from the cdh-local motivic cohomology to motivic cohomology.

\begin{thm}\label{thm:MotToCDHlsurj}
If $l \in \Zb$ is invertible in $k$, then the cdh-sheafification map
\[
\Zb(i)^\mot \rev{\to} \Zb(i)^\cdh ,
\]
induces a surjection
\[
H^{i,j}_\mot(X; \Zb)_l \tto H^{i,j}_\cdh (X ; \Zb)_l
\]
between $l$-torsion subgroups in the motivic cohomology and the cdh-local motivic cohomology for any derived $k$-scheme $X$.
\end{thm}

Moreover, the motivic complexes satisfy the following structural properties \cite[Theorem~1.1, Remark~5.25]{elmanto-morrow}.

\begin{thm}
Let $X$ be a \rev{qcqs} derived $k$-scheme. Then the motivic complexes satisfy the following properties:
\begin{enumerate}
    \item \emph{Relationship to Picard group:} there is a natural isomorphism $\mathrm{Pic}(X) \cong H^{2,1}_\mot(X;\Zb)$;
    \item \emph{Projective bundle formula:} pullback, multiplied with powers of the first Chern class of the tautological line bundle on $\Pb^n_X$, induces an equivalence
    \[
    \bigoplus_{i = 0}^{n} \Zb(j-i)^\mot(X)[-2i] \stackrel{\sim}{\to} \Zb(j)^\mot(\Pb^n_X);
    \]
    \item \emph{Derived blowup excision:} if $Z \hook X$ is a derived regular embedding, then the derived blowup square in the sense of \cite{khan-rydh} induces a Cartesian square
    \[
    \begin{tikzcd}
        \Zb(i)^\mot(E) & \arrow[l] \Zb(i)^\mot(\bl_Z(X)) \\
        \Zb(i)^\mot(Z) \arrow[u] & \Zb(i)^\mot(X) \arrow[u] \arrow[l]
    \end{tikzcd}
    \]
    in $D(\Zb)$.
\end{enumerate}
\end{thm}

\begin{rem}\label{rem:DerivedPBF}
Elmanto and Morrow state the projective bundle formula only for classical schemes \cite[Theorem~5.24]{elmanto-morrow}. However, the result is true for \rev{qcqs} derived schemes as well. To see this, recall that by passing to associated graded objects in \cite[Construction~4.38]{elmanto-morrow}, we obtain cartesian squares
\[
\begin{tikzcd}
    \Zb(j)^\mot(X) \arrow[d] \arrow[r] & R\Gamma (X, \widehat{L\Omega}^{\geq j}_{- / \Qb}) \arrow[d]\\
    \Zb(j)^\mot(X_\cl) \arrow[r] & R\Gamma (X_\cl, \widehat{L\Omega}^{\geq j}_{- / \Qb})
\end{tikzcd}
\quad \mathrm{and} \quad
\begin{tikzcd}
    \Zb(j)^\mot(X) \arrow[d] \arrow[r] & \Zb_{\rev{p}}(j)^\syn(X) \arrow[d] \\
    \Zb(j)^\mot(X_\cl) \arrow[r] & \Zb_{\rev{p}}(j)^\syn(X_\cl)
\end{tikzcd}
\]
of spectra (first if $X$ has characteristic 0 and the second if $X$ has positive characteristic \rev{$p$}). In either case, all other terms in the square except the top left corner are known to satisfy \rev{the} projective bundle formula. For motivic cohomology of classical schemes this is \cite[Theorem~5.24]{elmanto-morrow}. For the Hodge-completed derived de Rham cohomology with its Hodge filtration this follows from the projective bundle formula for derived Hodge cohomology, which in turn follows from the computation of the Hodge cohomology of $\Pb^n_\Qb$, as derived Hodge cohomology satisfies Künneth formula \cite[Remark~B.3]{bhatt-lurie:apc}. For syntomic cohomology this follows from e.g. \cite[Proposition~6.7]{AHI:atiyah} by passing to associated graded objects. 
\end{rem}

\subsection{Gysin maps along derived regular embeddings}\label{subsect:LongkeGysin}

In \cite{Tang-Gysin}, Longke Tang constructs \textit{Gysin pushforward maps} along derived regular embeddings for essentially all cohomology theories on derived schemes that satisfy Nisnevich descent and excision in derived blowups \cite{khan-rydh}. Below, we formulate a consequence of Tang's result that is sufficient for the purposes of our paper. We will denote by $\MS^\dbe_S$ a version of motivic spectra over base $S$ \cite{AHI} that represents those cohomology theories on almost finite presentation (afp) derived $S$-schemes that satisfy Nisnevich descent and excision in derived blowups in the sense of Khan--Rydh \cite{khan-rydh} (see Appendix~\ref{sect:MSdbe}). As the motivic cohomology constructed by Elmanto--Morrow is representable in $\MS^\dbe_k$ (Appendix~\ref{sect:Hrep}), we obtain Gysin pushforwards along derived regular embeddings for it.

\begin{thm}[Tang]
Let $E \in \MS^\dbe_S$ be an oriented and homotopy commutative ring object. Then, if $i \colon Z \hook X$ is a derived regular embedding of virtual codimension $c$, there exists \emph{Gysin pushforward} homomorphisms
\[
i_!\colon E^{a,b}(Z) \to E^{a+2c,b+c}(X).
\]
Moreover, the Gysin maps satisfy the following properties:
\begin{enumerate}
    \item \emph{Functoriality:} $\Id_! = \Id$ and if $i \colon Z \hook X$ and $j \colon X \hook Y$ are derived regular embeddings of constant virtual codimension,\footnote{\rev{A derived regular embedding $i \colon Z \hook X$ is of \emph{constant virtual codimension $r$} if, locally around $Z$, it is the derived vanishing locus of $r$ functions on $X$. It is of \emph{constant virtual codimension} if it is of constant virtual codimension $r$ for some $r \in \Nb$.}} then $(j \circ i)_! = j_! \circ i_!$.
    \item \emph{Base change:} If
    \[
    \begin{tikzcd}
        Z' \arrow[d]{}{p'} \arrow[r,hook]{}{i'} & X' \arrow[d]{}{p} \\
        Z \arrow[r,hook]{}{i} & X
    \end{tikzcd}
    \]
    is a Cartesian square of afp derived $S$-schemes and $i$ (hence $i'$) is a derived regular embedding, then $p^* \circ i_! = i'_! \circ p'^*$.
    \item \emph{Projection formula:} The formula $i_!(i^*(\alpha) \cdot \beta) = \alpha \cdot i_!(\beta)$ holds for all $\alpha \in E^{a,b}(X)$ and $\beta \in E^{a',b'}(Z)$.
    \item \emph{Euler classes:} If $i \colon Z \hook X$ is the derived vanishing locus of a global section of a vector bundle $\Ec$ of rank $r$ on $X$, then $i_!(1_Z) = c_r(\Ec) \in E^{2r,r}(X)$\rev{, where $c_r(\Ec)$ is the top Chern class which is defined in the standard fashion using projective bundle formula, see e.g. \cite[Definition~4.4.1]{annala-iwasa:MotSp}.}
    \item \emph{Naturality:} If $\theta: E \to F$ is an orientation preserving map of homotopy commutative ring objects in $\MS_S^\dbe$, then the squares
    \[
        \begin{tikzcd}
             E^{a,b}(Z) \arrow[r]{}{i_!} \arrow[d]{}{\theta} & E^{a+2c,b+c}(X) \arrow[d]{}{\theta} \\
             F^{a,b}(Z) \arrow[r]{}{i_!} & F^{a+2c,b+c}(X)
        \end{tikzcd}
    \]
    commute.
\end{enumerate}
\end{thm}

\subsection{Gysin maps along projective derived-lci maps}\label{subsect:LongkeGysinProj}

Here, we recall how to extend the Gysin maps along derived regular embeddings to Gysin maps along all \textit{projective derived-lci maps}, i.e., maps $f \colon Z \to X$ that can be factored as
\[
Z \stackrel{i}{\hook} \Pb^n_X \stackrel{p}{\to} X
\]
where $p$ is the canonical projection and $i$ is a derived regular embedding. The idea is quite simple: Gysin pushforwards along projections $p\colon \Pb^n_X \to X$ can be constructed by hand, because oriented cohomology theories satisfy the projective bundle formula. One can then define $f_!$ as the composition $p_! \circ i_!$, and verify that it does not depend on the chosen factorization of $f$ into $i$ and $p$. Our treatment follows closely that of \cite[Section~5]{deglise:2008}. For simplicity, we assume the base derived scheme $S$ to be qcqs. Throughout the section, $E \in \MS^\dbe_S$ will denote an oriented and homotopy commutative ring object, and $E^*$ will denote the graded ring with graded pieces $E^i = E^{2i,i}(S)$.

The first step is to define the Gysin morphism along projections $p: \Pb^n_X \to X$ of relative projective spaces. The coefficients defining $p_!$ depend on the formal group law of $E$, which we recall next. 

\begin{defn}
By \cite[4.4.2]{annala-iwasa:MotSp} there exists a unique power series
\[
F_E(x,y) = \sum_{i,j \geq 0} a_{ij} x^i y^j \in E^*[[x,y]]
\]
with $a_{00} = 0$, $a_{10} = a_{01} = 1$, and $a_{ij} \in E^{1-i-j}$, and such that for all qcqs afp derived $S$-schemes $X$, and line bundles $\Ls_1$ and $\Ls_2$ on $X$, the equality
\[
c_1(\Ls_1 \otimes \Ls_2) = F_E(c_1(\Ls_1), c_1(\Ls_2)) \in E^{2,1}(X)
\]
holds. Note that the right hand side of the above equation is well defined, as first Chern classes of line bundles are nilpotent under the above assumptions on $X$ (see the proof of \cite[Lemma~3.1.8]{annala-iwasa:MotSp}).

The power series $F_E$ is a \emph{formal group law} i.e., it has the following properties:
\begin{enumerate}
    \item \emph{Unitality}: $F_E(x,0) = F_E(0,x) = x$.
    \item \emph{Commutativity}: $F_E(x,y) = F_E(y,x)$.
    \item \emph{Associativity}: $F_E(F_E(x,y),z) = F_E(x, F_E(y,z))$.
\end{enumerate}
\end{defn}

The coefficients that are used to define $p_!$ are closely related to the class of the diagonal embedding of projective spaces. Next, we identify these classes explicitly in terms of the formal group law of $E$.

\begin{lem}
We have the equality
\[
\Delta_!(1_{\Pb^n}) = \sum_{i,j = 0}^n a_{1,i+j-n} x_1^i x_2^j \in E^{2n,n}(\Pb^n \times \Pb^n),
\]
where $\Delta \colon \Pb^n \to \Pb^n \times \Pb^n$ is the diagonal embedding, $x_1 = c_1(\Oc(1,0))$, and $x_2 = c_1(\Oc(0,1))$. 
\end{lem}
\begin{proof}
Consider the canonical exact sequence
\[0 \to \Oc(-1,0) \stackrel{\iota_1}{\hook} \Oc^{n+1} \stackrel{\psi_1}{\tto} Q_1 \to 0\]
of vector bundles on $\Pb^n \times \Pb^n$, and denote by $\phi$ the composition $\Oc(0,-1) \stackrel{\iota_2}{\hook} \Oc^{n+1} \stackrel{\psi_1}{\tto} Q_1$. Then $\phi$ corresponds to a global section of $Q_1(0,1)$, whose derived vanishing locus is the diagonal $\Delta \colon \Pb^n \hook \Pb^n \times \Pb^n$. Thus, $\Delta_!(1_{\Pb^n}) = c_n(Q_1(0,1))$.

By the basic properties of Chern classes,
\[
c_n(Q_1(0,1)) = \sum_{i,j = 0}^n \eta^{(n)}_{ij} x_1^i x_2^j 
\]
for some constants $\eta^{(n)}_{ij} \in E^{n - i - j}$ that only depend on the formal group law $F_E$. It therefore suffices to check the equality $\eta^{(n)}_{ij} = a_{1,i+j-n}$ in the universal case where $F$ is the formal group law of the Lazard ring \cite{lazard:1955}. Thus, it suffices to verify that the class of the diagonal embedding in the complex bordism group $\mathrm{MU}_{2n}(\CP^n \times \CP ^n)$ is given by the same formula. This can be proven using the explicit formulas for the structure constants of the Hopf algebra structure of $\MU_*(\CP^\infty)$ provided in Appendix~\ref{sect:MUCoprod} (Proposition~\ref{prop:MUCoprodFormula}).
\end{proof}

For notational simplicity, set $\eta_i = a_{1i}$, and consider the pair of inverse matrices
\[
M_n = 
\begin{bmatrix}
0 & \cdots & 0 & 1 \\
\vdots & \iddots & \iddots & \eta_1 \\
0 & \iddots & \iddots & \vdots \\
1 & \eta_1 & \cdots & \eta_n 
\end{bmatrix}
\qquad 
M_n^{-1} = 
\begin{bmatrix}
\eta'_n & \cdots & \eta'_1 & 1 \\
\vdots & \iddots & \iddots & 0 \\
\eta'_1 & \iddots & \iddots & \vdots \\
1 & 0 & \cdots &0 
\end{bmatrix}
\]
By definition, the equality
\begin{equation}\label{eq:EtaPairing}
    \sum_{i = 0}^n \eta_i \eta'_{n-i} = 
    \begin{cases}
    1 & \text{if $n=0$;} \\
    0 & \text{otherwise;}
    \end{cases}
\end{equation}
holds in $E^*$.

We are now ready to define Gysin maps along projection maps. Note that by the projective bundle formula, each element $\alpha \in E^{a,b}(\Pb^n_X)$ can be uniquely expressed in the form
\[
\alpha = p^*(\alpha_0) + p^*(\alpha_1) \cdot c_1(\Oc(1)) + \cdots + p^*(\alpha_n) \cdot c_1(\Oc(1))^n
\]
where $\alpha_i \in E^{a-2i,b-i}(X)$.

\begin{defn}[Gysin maps along $p \colon \Pb^n_X \to X$]
We define
\begin{equation*}
    p_!(\alpha) := \alpha_n + \alpha_{n-1} \cdot \eta'_1 + \cdots + \alpha_0 \cdot  \eta'_n \in E^{a-2n,b-n}(X).
\end{equation*}
This formula gives rise to group homomorphisms $p_!\colon E^{a,b}(\Pb^n_X) \to E^{a-2n,b-n}(X)$.
\end{defn}

\begin{lem}\label{lem:ProjGysinProps}
The Gysin maps along projections satisfy the following properties:
\begin{enumerate}
    \item \emph{Base change:} Given \rev{a} Cartesian square
    \[
    \begin{tikzcd}
        \Pb^n_Y \arrow[d]{}{f'} \arrow[r]{}{p'} & Y\arrow[d]{}{f} \\
        \Pb^n_X \arrow[r]{}{p} & X
    \end{tikzcd}
    \]
    then $f^* \circ p_! = p'_! \circ f'^*$.
    \item \emph{Projection formula:}  The formula $p_!(p^*(\alpha) \cdot \beta) = \alpha \cdot p_!(\beta)$ holds for all $\alpha \in E^{a,b}(X)$ and $\beta \in E^{a',b'}(\Pb^n_X)$.
    \item \emph{Double projections:} Given a Cartesian square
    \[
    \begin{tikzcd}
        \Pb^m_{\rev X} \times_{\rev X} \Pb^n_X \arrow[d]{}{p'} \arrow[r]{}{q'} & \Pb^n_X\arrow[d]{}{p} \\
        \Pb^m_X \arrow[r]{}{q} & X
    \end{tikzcd}
    \]
    then $q_! \circ p'_! = p_! \circ q'_!$.
    \item \emph{Compatibility with Gysin maps along regular embeddings:} Given a Cartesian square
    \[
    \begin{tikzcd}
        \Pb^n_Z \arrow[d,hook]{}{i'} \arrow[r]{}{p'} & Z\arrow[d,hook]{}{i} \\
        \Pb^n_X \arrow[r]{}{p} & X
    \end{tikzcd}
    \]
    then $i_! \circ p'_! = p_! \circ i'_!$.
    \item \emph{Naturality:} If $\theta: E \to F$ is an orientation preserving map of homotopy commutative ring objects in $\MS_S^\dbe$, then the squares
    \[
        \begin{tikzcd}
             E^{a,b}(\Pb^n_X) \arrow[r]{}{p_!} \arrow[d]{}{\theta} & E^{a-2n,b-n}(X) \arrow[d]{}{\theta} \\
             F^{a,b}(\Pb^n_X) \arrow[r]{}{p_!} & F^{a-2n,b-n}(X)
        \end{tikzcd}
    \]
    commute.
\end{enumerate}
\end{lem}
\begin{proof}
Claim (1) follows from naturality of Chern classes. Claims (2) and (3) are obvious. Claim (4) follows from naturality of Chern classes, base change, and projection formula for Gysin pushforwards along derived regular embeddings. Claim (5) follows from the fact that $\theta$ sends the coefficients of the formal group law of $E$ to the coefficients of the formal group law of $F$.
\end{proof}

Our next goal is to verify that the composition $p_! \circ i_!$ is independent of the factorization of $f \colon Z \to X$ into a derived regular embedding $i \colon Z \hook \Pb^n_X$ into a relative projective space $p \colon \Pb^n_X \to X$.

\begin{lem}\label{lem:GysCompatibilityWithSections}
If $i \colon X \hook \Pb^n_X$ is a section of $p \colon \Pb^n_X \to X$, then $p_! \circ i_! = \Id$.
\end{lem}
\begin{proof}
Let $\nu: X \to \Pb^n$ be the $\Pb^n$-component of $i$. We then have a cartesian diagram
\[
\begin{tikzcd}
X \arrow[r,hook]{}{i} \arrow[d]{}{\nu} & \Pb^n \times X \arrow[d]{}{\Id \times \nu} \arrow[r]{}{p} & X \arrow[d]{}{\nu} \\ 
\Pb^n \arrow[r,hook]{}{\Delta} & \Pb^n \times \Pb^n \arrow[r]{}{p'} & \Pb^n
\end{tikzcd}
\]
of derived schemes. By Eq.~\eqref{eq:EtaPairing}, $p'_! \circ \Delta_!(1_{\Pb^n}) = 1_{\Pb^n}$. By base change, this implies that $p_! \circ i_! (1_X) = 1_X$. Finally, by projection formula
\begin{align*}
    p_! (i_! (\alpha)) &= \alpha \cdot p_! ( i_!(1_X)) \\
    &= \alpha
\end{align*}
for all $\alpha \in E^{a,b}(X)$, finishing the proof of the claim.
\end{proof}

Thus, we can verify that the Gysin pushforward of a projective derived-lci map is independent of the chosen factorization.

\begin{lem}
If
\[
\begin{tikzcd}
    Z \arrow[r,hook]{}{i} \arrow[d,hook]{}{j} & \Pb^n_X \arrow[d]{}{p} \\
    \Pb^m_X \arrow[r]{}{q} & X
\end{tikzcd}
\]
commutes, then $p_! \circ i_! = q_! \circ j_!$.
\end{lem}
\begin{proof}
Consider the diagram
\[
\begin{tikzcd}
    Z \arrow[rd,hook]{}[description]{e} \arrow[rrd, hook, bend left]{}{i} & & \\
    & \Pb^m \times \Pb^n \times X \arrow[r]{}{q'} \arrow[d]{}{p'} & \Pb^n \times X \arrow[d]{}{p} \\
    & \Pb^m \times X \arrow[r]{}{q} & X.
\end{tikzcd}
\]
By symmetry and Lemma~\ref{lem:ProjGysinProps}(3), it suffices to show that $i_! = q'_! \circ e_!$. To do so, consider the diagram
\[
\begin{tikzcd}
    Z \arrow[rrd, hook, bend left]{}{e} \arrow[rd,hook]{}[description]{\tilde e}  &  & \\
    & \Pb^m \times Z \arrow[r,hook]{}{i'} \arrow[d]{}{q''} & \Pb^m \times \Pb^n \times X \arrow[d]{}{q'} \\
    & Z \arrow[r,hook]{}{i} & \Pb^n \times X.
\end{tikzcd}
\]
As $\tilde e$ is a section of $q''$, we have $q''_! \circ \tilde e_! = \Id$ \rev{by Lemma~\ref{lem:GysCompatibilityWithSections}}. Moreover, $e_! = i'_! \circ \tilde e_!$ by functoriality of Gysin maps along regular embeddings. Thus, the desired equality follows from the fact that $q'_! \circ i'_! = i_! \circ q''_!$ (Lemma~\ref{lem:ProjGysinProps}(4)). 
\end{proof}

\begin{defn}[Gysin maps along projective derived-lci maps] If $f \colon X \to Y$ is a projective derived-lci map, then we define
\[
f_! := p_! \circ i_!,
\]
where $X \stackrel{i}{\hook} \Pb^n_Y \stackrel{p}{\to} Y$ is any factorization of $f$ into a derived regular embedding into a projective space over $Y$.
\end{defn}

Next, we prove that Gysin pushforwards along projective derived-lci maps are compatible with compositions.

\begin{lem}\label{lem:GysinAndProj}
If 
\[
\begin{tikzcd}
    X \arrow[r]{}{f'} \arrow[rd]{}{f} & \Pb^n_Y \arrow[d]{}{p} \\
    & Y
\end{tikzcd}
\]
is a commutative diagram of projective derived-lci maps, then $p_! \circ f'_! = f_!$.
\end{lem}
\begin{proof}
Denote by $g$ the composition $X \stackrel{f'}{\to} \Pb^n_Y \to \Pb^n$, and let $h \colon X \to \Pb^m$ be such that $(h,f)\colon X \hook \Pb^m \times Y$ is a derived regular embedding. Then we have a commutative diagram
\[
\begin{tikzcd}
    X \arrow[rrd, bend left]{}{f' = (g,f)} \arrow[rd,hook]{}[description]{(g,h,f)} \arrow[rdd, hook, bend right]{}[swap]{i = (h,f)} &&\\
    &\Pb^n \times \Pb^m \times Y \arrow[r] \arrow[d] & \Pb^n \times Y \arrow[d]{}{p} \\
    &\Pb^m \times Y \arrow[r]{}{q} & Y.
\end{tikzcd}
\]
Gysin pushforwards along the triangles commute by definition, and along the square by Lemma~\ref{lem:ProjGysinProps}(3). Thus, $f_! = q_! \circ i_! = p_! \circ f'_!$, as desired.
\end{proof}

\begin{lem}
If $f \colon Z \to X$ and $g \colon X \to Y$ are projective derived-lci maps, then $(g \circ f)_! = g_! \circ f_!$.
\end{lem}
\begin{proof}
Consider the commutative diagram
\[
\begin{tikzcd}
    Z \arrow[r,hook]{}{i} \arrow[rd]{}[swap]{f} &\Pb^n \times X \arrow[r,hook]{}{j} \arrow[d] &\Pb^n \times \Pb^m \times Y \arrow[d]{}{p} \\
    &X \arrow[r,hook]{}{} \arrow[rd]{}[swap]{g} &\Pb^m \times Y \arrow[d]{}{q} \\
    &&Y.
\end{tikzcd}
\]
By the definition of Gysin pushforwards and by Lemma~\ref{lem:GysinAndProj},
\[
(g \circ f)_! = q_! \circ p_! \circ j_! \circ i_!.
\]
As the Gysin pushforwards along the two triangles commute by definition, and along the square by Lemma~\ref{lem:ProjGysinProps}(4), we have that
\[
q_! \circ p_! \circ j_! \circ i_! = g_! \circ f_!,
\]
proving the claim.
\end{proof}

We record the content of this section as the following theorem.

\begin{thm}\label{thm:GysinProps}
Let $E \in \MS^\dbe_S$ be an oriented and homotopy commutative ring object. Then, if $f \colon Z \to X$ is a projective derived-lci map of virtual relative dimension $c$, there exists \emph{Gysin pushforward} homomorphisms
\[
f_!\colon E^{a,b}(Z) \to E^{a-2c,b-c}(X).
\]
Moreover, these maps satisfy the following properties:
\begin{enumerate}
    \item \emph{Functoriality:} $\Id_! = \Id$ and if $f \colon Z \to X$ and $g \colon X \to Y$ are projective derived-lci, then $(f \circ g)_! = f_! \circ g_!$.
    \item \emph{Base change:} If
    \[
    \begin{tikzcd}
        Z' \arrow[d]{}{p'} \arrow[r]{}{f'} & X' \arrow[d]{}{p} \\
        Z \arrow[r]{}{f} & X
    \end{tikzcd}
    \]
    is a Cartesian square of afp derived $S$-schemes and $f$ is projective and derived-lci, then $p^* \circ f_! = f'_! \circ p'^*$.
    \item \emph{Projection formula:} If $f: Z \to X$ is projective and derived-lci, then the formula $f_!(f^*(\alpha) \cdot \beta) = \alpha \cdot f_!(\beta)$ holds for all $\alpha \in E^{a,b}(X)$ and $\beta \in E^{a',b'}(Z)$.
    \item \emph{Euler classes:} If $i \colon Z \hook X$ is the derived vanishing locus of a global section of a vector bundle $\Ec$ of rank $r$ on $X$, then $i_!(1_Z) = c_1(\Ec) \in E^{2r,r}(X)$.
    \item \emph{Naturality:} If $\theta: E \to F$ is an orientation preserving map of homotopy commutative ring objects in $\MS_S^\dbe$, then the squares
    \[
        \begin{tikzcd}
             E^{a,b}(Z) \arrow[r]{}{f_!} \arrow[d]{}{\theta} & E^{a-2c,b-c}(X) \arrow[d]{}{\theta} \\
             F^{a,b}(Z) \arrow[r]{}{f_!} & F^{a-2c,b-c}(X)
        \end{tikzcd}
    \]
    commute. \qed
\end{enumerate}
\end{thm}

\subsection{Precobordism ring $\PCob^*$}\label{subsect:PCob}

Universal precobordism rings $\PCob^*(X)$ are an algebro-geometric versions of complex cobordism rings in topology. They were extensively studied by the first-named-author in his thesis  \cite{annala-thesis}.

\begin{defn}
Let $X$ be a derived scheme. Then the \emph{(universal) precobordism ring} $\PCob^*(X)$ is the graded ring defined as follows:
\begin{itemize}
    \item[(Gen)] the group $\PCob^i(X)$ is generated by formal sums of equivalence classes $[V \to X]$ of projective derived-lci maps of relative virtual dimension $-i$; 
\end{itemize}
subject to the following relations:
\begin{itemize}
    \item[(Add)] $[V_1 \sqcup V_2 \to X] = [V_1 \to X] + [V_2 \to X] \in \PCob^i(X)$;
    \item[(Cob)] if $\pi\colon W \to \Pb^1 \times X$ is projective derived-lci map of virtual relative dimension $i$ and 
    \begin{enumerate}
        \item $V_0$ is the derived fibre of $\pi$ over $0 \times X$;
        \item $A + B$ is an expression of the derived fibre $V_\infty$ of $\pi$ over $\infty \times X$ as the sum of two effective virtual Cartier divisors;
    \end{enumerate}
    then
    \[
        [V_0 \to X] = [A \to X] + [B \to X] - [\Pb_{A \cap B}(\Oc(A) \oplus \Oc) \to X]
    \]
    in $\PCob^i(X)$, where $A \cap B$ denotes the derived intersection of $A$ and $B$ on $W$.
\end{itemize}
The multiplicative structure on $\PCob^*(X)$ is given by
\[
[V_1 \to X] \cdot [V_2 \to X] = [V_1 \times_X V_2 \to X],
\]
where $\times_X$ denotes the derived fiber product over $X$.
\end{defn}

\begin{rem}
The relation (Cob) combines two relations within. The first one is the \emph{naive cobordism relation}, i.e., that there is an equality 
\[
[V_0 \to X] = [V_\infty \to X]
\]
whenever $V_0$ and $V_\infty$ are two fibers of a projective derived-lci map $W \to \Pb^1 \times X$. The second one is a \emph{decomposition relation} which enforces a geometric counterpart of the formal group law in a special case. In fact, one can show that the second relation is essentially equivalent to enforcing the formal group law, see \cite[Theorem~3.16]{annala-pre-and-cob}.
\end{rem}

\begin{rem}[Algebraic cobordism]
One obtains the \emph{algebraic cobordism ring} of $X$, $\Omega^*(X)$, as a further quotient of $\PCob^*(X)$ by the so called \emph{snc-relations} (see \cite{annala-pre-and-cob, annala-base-ind-cob}). However, since the definition of $\PCob^*(X)$ is more concrete, and because the distinction between $\PCob^*(X)$ and $\Omega^*(X)$ will not be important for us, we will only work with the precobordism rings in this article.
\end{rem}

The rings $\PCob^*(X)$ form a multiplicative cohomology theory in the sense that there are \emph{pullback homomorphisms} $f^* \colon \PCob^*(X) \to \PCob^*(Y)$ along arbitrary maps $f \colon Y \to X$ given by
\[
f^*[V \to X] = [V \times_X Y \to Y].
\]
Moreover, if $g \colon X \to X'$ is projective and derived-lci of relative virtual dimension $r$, then there exists an additive \emph{Gysin pushforward map} $g_! \colon \PCob^*(X) \to \PCob^{*-r}(X')$ given by
\[
g_! [V \stackrel{h}{\to} X] = [V \stackrel{g \circ h}{\to} X'].
\]
Both pullbacks and Gysin pushforwards are functorial. Moreover, they satisfy the following properties:
\begin{enumerate}
    \item \emph{base change}: if 
    \[
    \begin{tikzcd}
        X' \arrow[r]{}{f'} \arrow[d]{}{g'} & Y' \arrow[d]{}{g} \\
        X \arrow[r]{}{f} & Y
    \end{tikzcd}
    \]
    is a Cartesian square of derived schemes and $g$ is projective and derived-lci, then
    \[
        f^* \circ g_! = g'_! \circ f'^*;
    \]
    \item \emph{projection formula}: if $g \colon X \to X'$ is projective and derived-lci, then
    \[
    g_!(g^*(\alpha) \cdot \beta) = \alpha \cdot g_!(\beta)
    \]
    for all $\alpha \in \PCob^*(X')$ and $\beta \in \PCob^*(X)$.
\end{enumerate}
In addition to these formal properties, the precobordism theory satisfies projective bundle formula, as we recall next.

\begin{defn}
Let $\Ec$ be a vector bundle of rank $r$ on $X$. Then, we define the \emph{Euler class} of $\Ec$ as
\[
c_r(\Ec) := i_!(1) \in \PCob^r(X),
\]
where $i \colon Z \hook X$ is the derived vanishing locus of a global section of $\Ec$. The Euler class is independent of the choice of the section \cite[Lemma 132]{annala-thesis}.
\end{defn}

The following is a special case of \cite[Theorem~166]{annala-thesis}.

\begin{thm}[Projective bundle formula]
Let $X$ be a derived scheme that is noetherian, has finite Krull dimension, and admits an ample line bundle. Let $\Ec$ be a vector bundle of rank $r$ on $X$, and let $\Oc(1)$ be the tautological line bundle on the projective bundle $\pi \colon \Pb(\Ec) \to X$. Then the map
\[
c_1(\Oc(1))^i \cdot \pi^* (-) \colon \bigoplus_{i=0}^{r-1}\PCob^{*-i}(X) \to \PCob^*(\Pb(\Ec))
\]
is an isomorphism.
\end{thm}

\subsection{Cycle maps from $\PCob^*$}\label{subsect:CycleMap}

Above, we recalled that the precobordism rings $\PCob^*(X)$ form a cohomological functor that admit Gysin pushforwards along projective derived-lci maps. Here, we prove that $\PCob^*$ is universal among such functors. Thus we obtain convenient conditions, under which it is possible to define a \emph{cycle map} from $\PCob^*$ to other cohomology theories. For the purposes of the rest of the article, the reader may take Corollary~\ref{cor:CycleMapsFromPCob} as a black box, and ignore the rest of the section.

We begin by introducing the following terminology to distinguish functors like $\PCob^*$ from the bigraded cohomology theories represented by motivic spectra.

\begin{defn}
Let $S$ be a noetherian derived scheme of finite Krull dimension. A \emph{weak oriented cohomology functor} over $S$ is a contravariant functor 
\[
E^* \colon \QProj^\op_S \to \Ring^*
\]
from quasi-projective derived $S$-schemes\footnote{A map of derived schemes $X \to Y$ is \emph{quasi-projective} if it is of almost finite presentation and admits a relatively ample line bundle, see \cite[Section~3.2]{annala-thesis}.} to commutative graded rings, together with Gysin pushforwards
\[
g_! \colon E^*(X) \to E^{*-r}(Y)
\]
along projective derived-lci morphisms of relative dimension $r$, satisfying the following properties:
\begin{enumerate}
    \item \emph{Additivity}: $E^*(\emptyset) \cong 0$ and the map $E^*(X_1 \sqcup X_2) \to E^*(X_1) \times E^*(X_2)$ given by the evident pullbacks is an isomorphism;
    \item Gysin pushforwards are functorial;
    \item Gysin pushforwards and pullbacks satisfy base change in Cartesian squares;
    \item Gysin pushforwards and pullbacks satisfy projection formula.
\end{enumerate}
\end{defn}

\begin{ex}\label{ex:WOCF:PCob}
The universal precobordism rings $\PCob^*$, restricted to $\QProj_S$ for a noetherian $S$ of finite Krull dimension, is an example of a weak oriented cohomology functor. Note that additivity follows from the relation (Add) in the construction of $\PCob^*$.
\end{ex}

\begin{ex}\label{ex:WOCF:MS}
Let $S$ be noetherian and of finite Krull dimension, and let $E \in \MS^\dbe_S$ be an oriented and homotopy commutative ring object. Then we define a weak oriented cohomology functor $E^* \colon \QProj_S \to \Ring^*$ by
\[
E^n(X) := E^{2n,n}(X).
\]
Note that additivity follows from the fact that $E$ has Zariski descent, and the other properties from the basic properties of Gysin maps established in Theorem~\ref{thm:GysinProps}.
\end{ex}

It turns out that the universal precobordism rings $\PCob^*$ give the universal weak oriented cohomology functor that satisfies projective bundle formula. To state this precisely, we need to define the ingredients of the statement.

\begin{defn}\label{def:WOCFEuler}
Let $E^*$ be a weak oriented cohomology functor over $S$, $X \in \QProj_S$, and let $\Ec$ be a vector bundle of rank $r$ on $X$. Then the \emph{Euler class} of $\Ec$ is defined as
\[
c_r(\Ec) := i_{0!}(1_{Z_0}) \in E^{r}(X),
\]
where $i_0 \colon Z_0 \hook X$ is the derived vanishing locus of the zero section of $\Ec$.\footnote{Note that if $\Ec$ has rank $r>0$, then the derived vanishing locus of its zero-section is a non-trivial derived regular embedding of codimension $r$.} 
\end{defn}

\begin{defn}
Let everything be as in Definition~\ref{def:WOCFEuler}. We say that the Euler classes of $E^*$ are \emph{independent} if the equality 
\[
c_r(\Ec) = i_{!}(1_{Z_s})
\]
holds, where $i \colon Z_s \hook X$ is the derived vanishing locus of a section $s$ of $\Ec$. In other words, for such a $E^*$ the cycle class associated to the derived vanishing locus of a vector bundle section depends only on the vector bundle, and not on the section $s$.
\end{defn}

\begin{ex}
Both the examples of form Example~\ref{ex:WOCF:PCob} and Example~\ref{ex:WOCF:MS} have independent Euler classes.
\end{ex}

\begin{defn}
Let $E^*$ be a weak oriented cohomology functor over $S$ that has independent Chern classes. Then $E^*$ \emph{satisfies the weak projective bundle formula} if for all $X \in \QProj_S$, the natural map
\[
E^*(X)[t]/(t^{n+1}) \to E^*(\Pb^n_X),
\]
defined by sending the degree-one element $t$ to $c_1(\Oc(1))$, is an isomorphism.\footnote{It follows from the independence of Euler classes that $c_1(\Oc(1))^{n+1} = 0 \in E^{n+1}(\Pb^n_X)$.}
\end{defn}

The more precise statement we want to make is that $\PCob^*$ is the universal weak oriented cohomology functor that satisfies the weak projective bundle formula. We still need to define the correct notion of transformation of such functors.
    
\begin{defn}
A \emph{natural transformation} $\vartheta \colon E^* \to F^*$ of weak oriented cohomology functors is a natural transformation of contravariant functors $\QProj^\op_S \to \Ring^*$ that is compatible with Gysin pushforwards in the sense that the squares
\[
\begin{tikzcd}
    E^*(X) \arrow[r]{}{\vartheta} \arrow[d]{}{g_!} & F^*(X) \arrow[d]{}{g_!} \\
    E^{*-r}(Y) \arrow[r]{}{\vartheta} & F^{*-r}(Y)
\end{tikzcd}
\]
commute for all $g \colon X \to Y$ projective and derived-lci of pure relative dimension, where $r$ is the relative dimension.
\end{defn}

\begin{ex}
An orientation preserving map $E \to F$ of homotopy commutative ring objects in $\MS_S^\dbe$ induces a natural transformation $E^* \to F^*$ of weak oriented cohomology functors by Theorem~\ref{thm:GysinProps}.
\end{ex}

The following theorem in due to Annala--Iwasa. It was recorded in the first-named-author's thesis \cite[Theorem~192(1)]{annala-thesis}.

\begin{thm}
Suppose $S$ is a noetherian derived scheme of finite Krull dimension that admits an ample line bundle, and let $E^*$ be a weak oriented cohomology functor over $S$ that satisfies the weak projective bundle formula. Then there exists a unique natural transformation $\vartheta \colon \PCob^* \to E^*$, and it is given by
\[
\vartheta[V \stackrel{f}{\to} X] = f_!(1_V) \in E^*(X).
\]
\end{thm}

For the convenience of the reader, we provide a black box reformulation of a special case of the above theorem.

\begin{cor}\label{cor:CycleMapsFromPCob}
Suppose $S$ is a noetherian derived scheme of finite Krull dimension that admits an ample line bundle, and let $E$ be a homotopy commutative ring object in $\MS_S^\dbe$. Then, for every $X \in \QProj_S$, there exists a natural \emph{cycle class map}
\begin{align*}
\cl \colon \PCob^n(X) &\to E^{2n,n}(X) \\
[V \stackrel{f}{\to} X] &\mapsto f_!(1_V)
\end{align*}
that is compatible with the graded ring structure, pullbacks, and Gysin pushforwards. 

Moreover, these cycle class maps are compatible with orientation preserving maps $\theta \colon E \to F$ of homotopy commutative ring objects in $\MS_S^\dbe$ in the sense that, for all $X$, the triangle
\[
\begin{tikzcd}
    \PCob^n(X) \arrow[r]{}{\cl} \arrow[rd]{}{\cl} & E^{2n,n}(X) \arrow[d]{}{\theta} \\
    & F^{2n,n}(X)
\end{tikzcd}
\]
commutes.
\end{cor}

\section{Singular varieties representing simplicial suspensions}

Here we construct singular varieties that represent, cdh and $\Ab^1$-locally, simplicial suspensions of smooth varieties. We begin by recalling Thomason's version of Jouanolou trick \cite[Proposition~4.4]{weibel:1989}.

\begin{thm}
Let $X$ be a qcqs scheme that admits an ample family of line bundles. Then there exists an affine-space bundle
\[
\wtil X \to X
\]
with the source $\wtil X$ an affine scheme.
\end{thm}

As any separated regular noetherian scheme has an ample family of line bundles \cite[II~2.2.7.1]{berthelot:1971}, the above result applies to all smooth varieties. 

\begin{cons}\label{cons:SSus}
Let $X$ be a smooth $k$-variety. Choose an affine affine-space bundle $\wtil X \to X$, set $X_0 := \wtil X$, and define $A_0$ as the ring of functions on $X_0$. 

We construct inductively a sequence of affine schemes $X_i$ by choosing a closed embedding $\iota_i \colon X_i \hook \Ab^{m_i}$, and defining $X_{i+1}$ to be the pushout
\[
X_{i+1} := \Ab^{m_i} \amalg_{X_{i}} \Ab^{m_i}.
\]
Note that pushouts of schemes along closed immersions exist \cite[Tag~0B7M]{stacks}. Moreover, $X_{i+1}$ is affine and the ring of functions on $X_{i+1}$ is  given by the fiber product
\[
A_{i+1} = k[x_1,\dots, x_{m_i}] \times_{A_i} k[x_1,\dots, x_{m_i}],
\]
where the two maps $k[x_1,\dots, x_{m_i}] \tto A_i$ are the pullbacks along the closed embedding $\iota_i$.
\end{cons}

Before using the above construction, we introduce the following ad hoc notion.

\begin{defn}\label{def:relcoh}
Let $X$ be a finite type $k$-scheme, and let $E \in \SH^\cdh_k$. Then the \textit{relative cohomology} of $E$ is defined as the functor that sends $X$ to
\[
\tilde E(X) := \mathrm{cofib} \big( E(k) \stackrel{\pi^*}{\to} E(X) \big),
\]
where $\pi\colon X \to \Spec(k)$ is the structure morphism. 
\end{defn}

\begin{rem}
If $X$ has a $k$-point $x$, then $x^*$ provides a splitting for $\pi^*$. Consequently, we obtain an isomorphism $\tilde E(X) \simeq E(X,x)$ where 
\[
E(X,x) := \mathrm{fib} \big( E(X) \stackrel{x^*}{ \to} E(x) \big)
\]
is the relative cohomology of $X$ pointed at $x$ in the usual sense.
\end{rem}

\begin{lem}\label{lem:SSus}
Let everything be as in Construction~\ref{cons:SSus}. Then, for all $E \in \SH^\cdh_k$, there are natural isomorphisms
\[
\tilde E^{i+n,j}(X_n) \cong \tilde E^{i,j}(X)
\]
of relative cohomology groups.
\end{lem}
\begin{proof}
By $\Ab^1$-invariance, the pullback along $X_0 \to X$ induces isomorphisms 
\[
\tilde E^{i,j}(X) \cong \tilde E^{i,j}(X_0).
\]
Moreover, for all $n \geq 0$, the pushout square
\[
\begin{tikzcd}
  X_n \arrow[r, hook] \arrow[d, hook] & \Ab^{m_n}  \arrow[d, hook] \\
  \Ab^{m_n} \arrow[r,hook]  & X_{n+1}
\end{tikzcd}
\]
is an abstract blowup square. By cdh-descent, this induces a cartesian square of relative cohomology spectra. As the relative cohomologies of affine spaces are trivial by $\Ab^1$-invariance, the connecting maps of the induced square induce isomorphisms $\tilde E^{i+1,j}(X_{n+1}) \cong \tilde E^{i,j}(X_n)$, from which the claim follows.
\end{proof}

\section{A singular counterexample to the motivic Steenrod problem}\label{sect:ObsCl}

The purpose of this section is to find a singular $k$ scheme $X$, together with an $l$-torsion class in the Chow ring $a \in \CH^n(X) = H^{2n,n}_\mot(X)$ that does not lift to the precobordism ring $\PCob^*(X)$, thus proving Theorem~\ref{thm:main}. We do this in two steps. Firstly, in Section~\ref{subsect:MGLObs}, we find a smooth variety $Y$ and a class $b \in H^{i,j}(Y ; \Zb)$ that does not lift to $\MGL^{i,j}(Y)$. Secondly, in Section~\ref{subsect:PCobObs}, we use the basic properties of the motivic cohomology groups to show that $b$ lifts to a class $\tilde b \in H^{2j,j}_\mot(Y_k; \Zb)$, where $Y_k$ is a singular variety essentially representing the $k$th simplicial suspension of $Y$ (see Construction~\ref{cons:SSus}), and $k = 2j-i$. This provides the desired example.

Throughout the section, $k$ is a field of characteristic $p \geq 0$ and $l$ is a prime number that is invertible in $k$. 

\subsection{Class in $H^{i,j}_\mot$ that does not lift to $\MGL^{i,j}$}\label{subsect:MGLObs}

Here we construct an example of a class in the motivic cohomology of a smooth variety that cannot be lifted to an $\MGL$-class. We do so by using well-known obstructions provided by motivic cohomology operations. The action of the cohomology operations is well understood on the classifying stack $\Brm \mu_l$. By leveraging this understanding, we find an unliftable classes in the motivic cohomology of an approximation of  $\Brm \mu_l \times \Brm \mu_l$ by a smooth variety.

We begin by recalling the motivic cohomology operations constructed by Voevodsky \cite{voevodsky:2003a}. The connecting maps of the fiber sequences
\[
\Zb(i) \stackrel{l \cdot}{\to} \Zb(i) \to \Zb(i)/l \qquad \mathrm{and} \qquad \Zb(i)/l \stackrel{l \cdot}{\to} \Zb(i)/l^2 \to \Zb(i)/l
\]
induce the \emph{Bockstein operations} $\tilde\beta \colon H^{i,j}_\mot(-; \Zb/l) \to H^{i+1,j}_\mot(-; \Zb)$ and $\beta \colon H^{i,j}_\mot(- ; \Zb/l) \to H^{i+1,j}_\mot(- ; \Zb/l)$. Note that $\beta$ is the composition of $\tilde\beta$ with reduction modulo $l$. Voevodsky also constructs \textit{motivic Steenrod power operations}
\[
P^r \colon H^{i,j}_\mot(- ; \Zb/l) \to H^{i+2r(l-1), j+r(l-1)}_\mot(-; \Zb/l).
\]

\begin{defn}\label{def:MotMil}
Define cohomology operations
\[
q_i = P^{l^{i-1}} P^{l^{i-2}} \cdots P^l P^1
\]
and 
\[
Q_i = q_i \beta - \beta q_i.
\]
The operation $Q_i$ is referred to as the \emph{$i$th motivic Milnor operation}.
\end{defn}

Motivic Milnor operations provide obstructions against lifting classes from motivic cohomology to algebraic cobordism. The following result is well known, but it seems hard to pinpoint an exact reference, so we provide a concise proof.

\begin{lem}\label{lem:ObsByMilnor}
Let $X$ be a smooth $k$-variety, and let $\alpha \in H^{i,j}_\mot (X; \Zb/l)$ be such that $Q_r(\alpha) \not = 0$. Then $\alpha$ is not in the image of $\MGL^{i,j}(X) \to H^{i,j}_\mot(X; \Zb/l)$.
\end{lem}
\begin{proof}
Let  $a_1, a_2, \dots \in L$ be a sequence of \textit{adequate generators} in the sense of \cite{hoyois:2017}, and let $n = l^r-1$. By \cite[Lemma~6.13]{hoyois:2017}, the square
\[
\begin{tikzcd}
    \MGL / a_{n} \arrow[r]{}{\delta} \arrow[d] & \Sigma^{2n+1,n} \MGL \arrow[d] \\
    H\Fb_l \arrow[r]{}{Q_r} & \Sigma^{2n+1,n} H\Fb_l
\end{tikzcd}
\]
commutes up to a multiplication by an element of $\Fb_l^\times$, where $H\Fb_l$ is the motivic spectrum representing motivic cohomology with mod-$l$ coefficients, and $\delta$ is the connecting map induced by taking quotient by $a_n$. Thus, if $Q_r(\alpha) \not = 0$, then either $\alpha$ does not lift to a class in $\MGL/a_n$, in which case it does not lift to $\MGL$ either, or it lifts to a class $\tilde \alpha$ in $\MGL/a_n$ such that $\delta(\tilde \alpha) \not = 0$. In both cases, $\alpha$ does not lift to $\MGL$.
\end{proof}

Using the above result, we find unliftable integral cohomology classes.

\begin{lem}\label{lem:ModlObsToIntObs}
Let $X$ be a smooth $k$-variety, and let $\alpha \in H^{i,j}_\mot(X; \Zb/l)$ be such that $Q_r \beta (\alpha) \not = 0$. Then $\tilde\beta(\alpha) \in H^{i+1,j}_\mot(X; \Zb)$ is not in the image of $\MGL^{i+1, j}(X) \to H^{i+1,j}_\mot(X; \Zb)$.
\end{lem}
\begin{proof}
As $\tilde\beta(\alpha)$ maps to $\beta(\alpha)$ in $H^{i+1,j}_\mot(X ; \Zb/l)$, and as the latter is not in the image of $\MGL^{i+1,j}(X)$ by Lemma~\ref{lem:ObsByMilnor}, neither is $\tilde\beta(\alpha)$.
\end{proof}

In order to construct explicit unliftable classes, we need a geometric object, the motivic cohomology of which, together with the action of the mod-$l$ motivic Steenrod algebra, is well understood. Here, we use the classifying stack $\Brm \mu_l$ for this purpose. We recall its basic cohomological properties that were established in \cite[Section~6]{voevodsky:2003a}.

\begin{thm}[Voevodsky]\label{thm:BmulProps}
Let $\Xfr$ be a presheaf of spaces on $\Sm_k$. Then there exists a natural isomorphism
\[
H^{*,*}_\mot (\Xfr \times  \Brm \mu_l , \Zb/l) = H^{*,*}_\mot (\Xfr, \Zb/l)[[u,v]] / (u^2 = \tau v + \rho u)
\]
where:
\begin{enumerate}
    \item $v \in H^{2,1}_\mot(\Brm \mu_l ; \Zb/l)$ is the first Chern class of the tautological line bundle $\Ls_{\mu_l}$ of $\Brm \mu_l$;
    \item $u \in H^{1,1}_\mot(k)$ is the unique element that pulls back to $0$ along $\Spec(k) \to \Brm \mu_l$, and satisfies
    \[
        \tilde\beta(u) = c_1(\Ls_{\mu_l}) \in H^{2,1}_\mot(\Brm \mu_l; \Zb);
    \]
    \item $\rho$ is the image of $[-1] \in H^{1,1}_\mot(k, \Zb)$;
    \item $\tau$ is defined as
    \[
    \tau = 
    \begin{cases}
        0 & \text{if $l \not = 2$ or $\mathrm{char}(k) = 2$;} \\
        [-1] \in \mu_2(k) = H^{0,1}_\mot(k; \Zb/2)  & \text{if $l = 2$ and  $\mathrm{char}(k) \not = 2$.} 
    \end{cases}
    \]
\end{enumerate}
Furthermore, the action of the motivic Steenrod algebra is determined by the action of the Steenrod algebra on $H^{*,*}_\mot (\Xfr, \Zb/l)$ and the formulas
\begin{align*}
    \beta(u) &= v; \\
    \beta(v) &= 0; \\
    P^i(v^n) &= {n \choose i} v^{il + n - i};  \\
    P^i(u) &= 0.
\end{align*}
\end{thm}

Applying the above result to the case where $\Xfr = \Brm \mu_l$ we find unliftable classes in the motivic cohomology of a stack.

\begin{lem}\label{lem:ObsClassOnStack}
We have that
\[
Q_1 \beta (u_1 \cdot u_2) = v_1^l \cdot v_2 - v_1 \cdot v_2^l \in H^{2(l+1), l+1}_\mot(\Brm \mu_l \times \Brm \mu_l ; \Zb/l)
\]
is non-zero.
\end{lem}
\begin{proof}
As $\beta^2 = 0$, we have that $Q_1 \beta = - \beta P^1 \beta$. As Bockstein is a graded derivation \cite[Section~8]{voevodsky:2003a}, we have 
\[
\beta (u_1 \cdot u_2) = v_1 \cdot u_2 - u_1 \cdot v_2.
\] 
Then, using the Cartan formula \cite[Proposition~9.7]{voevodsky:2003a} and Theorem~\ref{thm:BmulProps}, we compute that 
\[
\beta P^1 (v_1 \cdot u_2 - u_1 \cdot v_2) = v_1^l \cdot v_2 - v_1 \cdot v_2^l.
\]
\rev{By Theorem~\ref{thm:BmulProps}, the classes $v_1^l \cdot v_2$ and $v_1 \cdot v_2^l$ are linearly independent over $\Zb/l$. Thus, $v_1^l \cdot v_2 - v_1 \cdot v_2^l$ is non-zero, proving the claim.}
\end{proof}

In order to obtain unliftable motivic cohomology classes on a smooth variety, we approximate $\Brm \mu_l$ by varieties. As $\Brm \mu_l$ is $\Ab^1$-homotopy equivalent to the complement $U_\infty$ of the zero-section of $\Vb_{\Pb^\infty}(\Oc(-l))$ \cite[Lemma~6.3]{voevodsky:2003a}, the complements $U_n$ of the zero-section of $\Vb_{\Pb^n}(\Oc(-l))$ provide natural candidates. Indeed the motivic cohomology of $U_n$ is closely related to that of $\Brm \mu_l$, as we verify next.

\begin{lem}\label{lem:BmulApprox}
Let $\Xfr$ be a presheaf of spaces on $\Sm_k$. Then the pullback along the canonical map $U_n \to \Brm \mu_l$ induces isomorphisms
\[
H^{i,j}_\mot(\Xfr \times \Brm \mu_l ; \Zb/l) \to  H^{i,j}_\mot(\Xfr \times U_n ; \Zb/l) 
\]
for all $i$ and all $j \leq n$.
\end{lem}
\begin{proof}
Indeed, denoting $x = c_1(\Oc(1))$, the cofiber sequence 
\[
\Sigma^{-2(n+1), -(n+1)} H\Fb_l^{\Pb^\infty} \xto{x^{n+1} \cdot} H\Fb_l^{\Pb^\infty} \to H\Fb_l^{\Pb^n}   
\]
induces by \cite[Section~6]{voevodsky:2003a}, a cofiber sequence
\[
\Sigma^{-2(n+1), -(n+1)} H\Fb_l^{\Brm \mu_l} \xto{v^{n+1} \cdot} H\Fb_l^{\Brm \mu_l} \to  H\Fb_l^{U_n}
\]
of spectra, where the exponential motivic spectrum $H\Fb_l^X \in \SH_k$ represents the motivic cohomology functor
\[
Y \mapsto H^{*,*}_\mot(Y \times X;\Zb/l).
\]
Thus, the claim follows from the fact that $H^{i,j}_\mot(\Yfr; \Zb/l) = 0$ if $j<0$ for any presheaf of spaces $\Yfr$ on $\Sm_k$.
\end{proof}

We are finally ready to find an explicit unliftable class on a smooth variety.

\begin{prop}\label{prop:ObsClassOnVariety}
Let $U_l$ be the compliment of the zero section in $\Vb_{\Pb^l}(\Oc(-l))$. Then the pullback of $u_1 \cdot u_2$ to $H^{2,2}_\mot(U_l \times U_l; \Zb/l)$ is a class on which the operation $Q_1 \beta$ does not vanish. 

In particular,
\[
\tilde\beta(u_1 \cdot u_2) \in H^{3,2}_\mot(U_l \times U_l; \Zb)
\]
is an $l$-torsion class that is not in the image of $\MGL^{3,2}(U_l \times U_l)$.
\end{prop}
\begin{proof}
Since cohomology operations commute with pullbacks, it suffices by Lemma~\ref{lem:ObsClassOnStack} to check that $v_1^l \cdot v_2 - v_1 \cdot v_2^l$ is not annihilated by the pullback to $U_l \times U_l$. This is immediate by the explicit characterization of the pullback provided by  Lemma~\ref{lem:BmulApprox}. The second claim follows from Lemma~\ref{lem:ModlObsToIntObs}.
\end{proof}

\subsection{Class in $\CH^n$ that does not lift to $\PCob^n$}\label{subsect:PCobObs}

Here, we use the unliftable class found in the preceding section in order to construct a class in $\CH^n$ that does not lift to $\PCob^n$. Recall that by \rev{Corollary~\ref{cor:CycleMapsFromPCob}} we have a natural commutative square
\[
\begin{tikzcd}
    \PCob^n(X) \arrow[r]{}{\cl} \arrow[d]{}{\cl} & H^{2n,n}_\mot(X;\Zb) \arrow[d] \\
    \MGL^{2n,n}_\cdh(X) \arrow[r]{}{\vartheta} & H^{2n,n}_\cdh(X;\Zb),
\end{tikzcd}
\]
where the right vertical map is induced by cdh-sheafification. The right vertical map induces a surjection on $l$-torsion subgroups (Theorem~\ref{thm:MotToCDHlsurj}). Thus, it suffices to find an $l$-torsion class \rev{in $H^{2n,n}_\cdh(X;\Zb)$} that does not lift along $\vartheta$.

Let $U_l \times U_l$ be as in Proposition~\ref{prop:ObsClassOnVariety}. \rev{By Proposition~\ref{prop:ObsClassOnVariety}, it has a $\cdh$-local motivic cohomology class $\tilde \beta(u_1 \cdot u_2)$ that does not lift along $\vartheta$, but this is not on the Chow diagonal. Next, we will ``move'' it to the Chow diagonal.} As \rev{$U_l \times U_l$} has a $k$-rational point $u \in U_l \times U_l$, the relative cohomology groups $\tilde E^{*,*}(U_l \times U_l)$ (Definition~\ref{def:relcoh}) can be identified with $E^{*,*}(U_l \times U_l, u) \subset E^{*,*}(U_l \times U_l)$ for any $E \in \SH_k^\cdh$. Moreover, if we let $X_{l,1}$ be the singular $k$-scheme constructed from $U_l$ using the procedure of Construction~\ref{cons:SSus}, there are natural isomorphisms (Lemma~\ref{lem:SSus})
\begin{align}\label{eq:SusIso}
\tilde E^{i,j}(U_l \times U_l) \cong \tilde E^{i+1,j}(X_{l,1}). 
\end{align}
Thus, the class $\tilde\beta(u_1 \cdot u_2) \in H^{3,2}_\mot(U_l \times U_l; \Zb)$ \rev{from Proposition~\ref{prop:ObsClassOnVariety}} maps to a class
\rev{\begin{equation}\label{eq:CdhNonLiftableClass}
x \in H^{4,2}_\cdh(X_{l,1}; \Zb).
\end{equation}} 
As the isomorphisms of Eq.~\eqref{eq:SusIso} are natural in $E$, the class $x$\rev{, exactly like the class $\tilde\beta(u_1 \cdot u_2)$,} is not in the image of $\vartheta$. As $x$ is $l$-torsion, it lifts to a class
\rev{\begin{equation}\label{eq:NonLiftableClass}
\tilde x \in H^{4,2}_\mot(X_{l,1}; \Zb).
\end{equation}} 
Thus, we have proven the main result of the article.

\begin{thm}
\rev{Let $X_{l,1}$ be the singular $k$-scheme obtained by applying Construction~\ref{cons:SSus} to the smooth variety $U_l \times U_l$, and let $\tilde x \in \CH^2(X_{l,1})$ be the class of Eq.~\eqref{eq:NonLiftableClass}}. Then the class $\tilde x \in \CH^2(X_{l,1})$ is an $l$-torsion class that is not in the image of $\cl \colon \PCob^2(X_{l,1}) \to \CH^2(X_{l,1})$. In particular, the motivic cohomology groups $H^{2*,*}_\mot(X_{l,1};\Zb[e^{-1}])$ \rev{and $H^{2*,*}_\mot(X_{l,1};\Zb_{(l)})$} do not have the Steenrod property, where $e$ is the characteristic exponent of $k$.
\end{thm}

\section{Discussion}

\subsection{Existence of a cycle theoretic model for the Chow ring} 

The Chow ring of a smooth variety can be defined via an explicit, cycle theoretic model: it is exactly the ring of formal sums of closed subvarieties up to rational equivalence. Above, we have showed that the Chow ring of a singular scheme $X$ need not be generated by pushforwards $f_!(1_V)$ of fundamental classes along projective derived-lci maps $f \colon V \to X$. Thus, we have ruled out one cycle theoretic model for the Chow ring:  the Chow ring is not isomorphic to the ring one obtains from the algebraic cobordism ring of Annala \cite{annala-thesis} by enforcing the additive formal group law. 

The question that remains is the existence of any cycle theoretic model for the Chow ring of singular varieties. Recently, Park~\cite{park:2021} has constructed complexes of cycles in the pursuit of a cycle theoretic model for the motivic cohomology of singular schemes. He calls the homology groups of these complexes \emph{yeni higher Chow groups}. Unfortunately, the yeni Chow groups do not have the correct relationship with algebraic $K$-theory \cite[Section~1.6.3]{elmanto-morrow}, so it is unclear if one should expect a cycle theoretic model for the Chow ring coming from the yeni Chow groups.

The question of the existence of cycle theoretic models is intimately tied to the existence of \emph{transfers} (pushforwards) for the Chow rings: if the Chow ring admits a pushforward map along $f\colon V \to X$, then one can associate to $f$ the cycle class $f_! (1_V) \in \CH^*(X)$. Moreover, the existence of transfers in $\CH^*$ is closely related to the existence of pullbacks for the Chow groups $\CH_*$: according to the bivariant philosophy \cite{fulton-macpherson, yokura:2009}, the class of morphisms along which $\CH^*$ admits pushforwards is exactly the class of proper morphisms along which $\CH_*$ admits pullbacks. Thus we would like to raise the following dual pair of questions.

\begin{quest}[Existence of pushforwards and pullbacks]
\

\begin{enumerate}
    \item What is the largest class of morphisms of $k$-schemes along which the Chow rings $\CH^*$ admit natural transfer maps?
    \item What is the largest class of morphisms of finite type $k$-schemes along which the Chow groups $\CH_*$ admit natural pullback maps?
\end{enumerate}
\end{quest}
The question (2) seems to be an old one (see e.g. \cite[Example~17.4.6, Example~18.3.17]{fulton:1998}), but unfortunately it has not attracted much attention thus far.

\subsection{Motivic Steenrod problem at the characteristic} 

In this work, we have not investigated whether or not the motivic cohomology groups $H^{2*,*}_\mot(X; \Zb_{(p)})$ have the Steenrod property, where $p>0$ is the characteristic of the base field $k$. If $X$ is smooth, then the motivic cohomology groups are isomorphic to Chow groups. Consequently, if resolution of singularities with blowups (or alterations whose degree is coprime to $p$) holds, $H^{2*,*}_\mot(X; \Zb_{(p)})$ has the Steenrod property, as every algebraic cycle is a $\Zb_{(p)}$-linear combination of proper pushforwards of smooth varieties mapping to $X$. For this reason, the motivic cohomology of smooth $k$-varieties is expected to satisfy the Steenrod property at the characteristic, even though it might be hard to prove. On the other hand, a negative answer for the Steenrod problem at the characteristic would give a disproof of resolution of singularities (see \cite{shin2023priori}).

\rev{In the original version of this article, we expressed an expectation that there should exist a \textit{singular} finite type $k$-scheme $X$ such that $H^{2*,*}_\mot(X;\Zb_{(p)})$ does not have the Steenrod property, and suggested that this could be proven using Primozic's motivic Steenrod operations at the characteristic \cite{primozic:2020}. Such an example has been recently constructed in the article \cite{annala-elmanto} whose main purpose is to give a new construction of motivic Steenrod operations at the characteristic, addressing the shortcomings of Primozic's construction. The strategy of Section~\ref{sect:ObsCl} works out for the most part, but some new complications arise. For instance, lifting a $p$-torsion class along the cdh sheafification map $H^{*,*}_\mot(X;\Zb) \to H^{*,*}_\cdh(X;\Zb)$ is a non-trivial question, as the cofiber of the cdh-sheafification is $p$-complete \cite{elmanto-morrow}, and can therefore have an abundance of $p$-torsion.}

%
%

\appendix

\section{Hopf algebra structure of $\MU_*(\CP^\infty)$}\label{sect:MUCoprod}

Here, we recall useful facts of the Hopf algebra structure of $\MU_*(\CP^\infty)$. In particular, we give explicit formulas for the product and coproduct in the basis of $\MU_*(\CP^\infty)$ given by finite dimensional complex projective spaces embedded linearly into $\CP^\infty$. Explicit formula for the coproduct gives an explicit formula for the $\MU_*$-classes of diagonal embeddings $\CP^n \to \CP^n \times \CP^n$ used in the main text. Explicit formula for the product gives explicit formulas for the $\MU_*$-classes of Segre embeddings $\CP^n \times \CP^m \to \CP^{nm + n + m}$, which can be used to answer a question of Lee--Pandharipande \cite{pandharipande:2012} (see Corollary~\ref{cor:LeePQuest}).

We recall the purely formal aspects of the Hopf algebra structure of $\MU_*(\CP^\infty)$ from \cite[Section~3]{ravenel-wilson}. Given a space $X$, $\alpha \in \MU^*(X)$, and $\beta \in \MU^*(X)$, we denote by $\langle \alpha, \beta \rangle$ the pushforward of the cap product $\alpha \frown \beta$ to $\MU_* := \MU_*(pt)$.

\begin{prop}
Complex (co)bordism has the following properties:
\begin{enumerate}
    \item $\MU^*(\CP^\infty) \cong \MU^* [[x]]$, where $x$ is the first Chern class of the tautological line bundle $\Oc(1)$ on $\CP^\infty$.
    \item $\MU_*(\CP^\infty)$ is a free $\MU_*$-module on classes $\beta_{j} \in \MU_{2j}(\CP^\infty)$, where $\beta_j$ are uniquely determined by 
    \[
    \langle x^i, \beta_j \rangle = \delta_{ij}.
    \]
    \item We have that $\MU_*(\CP^\infty \times \CP^\infty) \cong \MU_*(\CP^\infty) \otimes_{\MU_*} \MU_*(\CP^\infty)$, and it is a free $\MU_*$-module on classes $\beta_i \otimes \beta_j$.
    \item We have 
    \[
        \Delta_*(\beta_n) = \sum_{i+j = n} \beta_i \otimes \beta_j \in \MU_*(\CP^\infty \times \CP^\infty).
    \]
\end{enumerate}
\end{prop}

As one can see from above, the coproduct has a very explicit formula in terms of the \emph{dual basis} of $\MU_*(\CP^\infty)$ given by the elements $\beta_j$. However, for our purposes, we would like to understand the Hopf algebra structure of $\MU_*(\CP^\infty)$ in terms of a more geometric basis, namely that which is given by the linear embeddings $\iota_n: \CP^n \hook \CP^\infty$.

\begin{defn}
Denote by $p_n \in \MU_{2n} (\CP^\infty)$ the complex bordism class corresponding to $\iota_n$.
\end{defn}

Clearly $x \frown p_n = p_{n-1}$. The following result explains the relationship between the $\beta_i$ and the $p_i$.

\begin{lem}\label{lem:MUBaseChange}
The equality
\begin{equation}\label{eq:MUBaseChange1}
\beta_n = \sum_{i=0}^n a_{1i} \cdot p_{n-i}    
\end{equation}
holds in $\MU_{2n}(\CP^\infty)$, where $a_{ij} \in \MU_{2(i+j-1)}$ are coefficients of the formal group law of $\MU$.
\end{lem}
\begin{proof}
For a closed (almost) complex manifold $X$ of complex dimension $n$, denote by $[X] \in \MU_{2n}$ its class in complex bordism. Pushing forward $c_{1}(\Oc(1,1)) \frown 1_{\CP^1 \times \CP^n}$ and applying the formal group law of $\MU$, we observe that
\[
[H_{1,n}] = [\CP^1] \cdot [\CP^{n-1}] + \sum_{i=0}^n a_{1i} [\CP^{n-i}],
\]
where $H_{1,n}$ is a Milnor hypersurface. As $[H_{1,n}] = [\CP^1 \times \CP^{n-1}] \in \MU_{2n}$ (see e.g. \cite{solomadin:2018} for an explicit proof), we have proven that, for $n>0$,
\[
\beta'_n := \sum_{i=0}^n a_{1i} \cdot p_{n-i}
\]
pushes forward to $0 \in \MU_{2n}$. As $x^i \frown \beta'_n = \beta'_{n-i}$, we have proven that
\[
\langle x^i, \beta'_j \rangle = \delta_{ij}.
\]
As this was the defining property of $\beta_j$, it follows that $\beta'_j = \beta_j$.
\end{proof}

The inverse change of basis also has an explicit formula. 

\begin{lem}
The equality
\begin{equation}\label{eq:MUBaseChange2}
p_n = \sum_{i=0}^n [\CP^i] \cdot \beta_{n-i}
\end{equation}
holds in $\MU_{2n}(\CP^\infty)$, where $[\CP^i] \in \MU_{2i}$ is the bordism class of the complex projective space.
\end{lem}
\begin{proof}
Note that $\sum_{i+j = k} [\CP^i] \cdot a_{1j}$ is the pushforward of $\beta_k$ to $\MU_{2k}$ by Lemma~\ref{lem:MUBaseChange}, and therefore
\[
\sum_{i+j = k} [\CP^i] \cdot a_{1j} = \delta_{k0}.
\]
We then verify that Eqs.~\eqref{eq:MUBaseChange1} and~\eqref{eq:MUBaseChange2} are inverses of each others by explicit computations:
\begin{align*}
    \sum_{i=0}^n [\CP^i] \cdot \sum_{j=0}^{n-i} a_{1j} \cdot p_{n-i-j} &= \sum_{i,j \geq 0} [\CP^i] \cdot a_{1j} \cdot p_{n-i-j} \\
    &= p_n
\end{align*}
and
\begin{align*}
    \sum_{i=0}^n a_{1i} \cdot \sum_{j=0}^{n-i} [\CP^j] \cdot \beta_{n-i-j} &= \sum_{i,j \geq 0} a_{1i} \cdot [\CP^j] \cdot \beta_{n-i-j} \\
    &= \beta_n,
\end{align*}
proving the claim.
\end{proof}

Next, we use the above results to find an explicit formula for the coproduct of $\MU_*(\CP^\infty)$ in terms of the $p_i$-basis.

\begin{prop}\label{prop:MUCoprodFormula}
The equality
\begin{equation}\label{eq:MUCoprod}
\Delta_*(p_n) = \sum_{i,j \geq 0} a_{1,i+j-n} \cdot p_{n-i} \otimes p_{n-j}
\end{equation}
holds in $\MU_*(\CP^\infty \times \CP^\infty)$.
\end{prop}
\begin{proof}
We compute that
\begin{align*}
\Delta_*(p_n) &= \sum_{k=0}^n [\CP^k] \cdot \Delta_*(\beta_{n-k}) \\
&= \sum_{i,j \geq 0}  [\CP^{n-i-j}] \cdot \beta_{i} \otimes \beta_{j} \\
&= \sum_{i = 0}^n \beta_i \otimes p_{n-i} \\
&= \sum_{i = 0}^n \sum_{l = 0}^i a_{1l} \cdot p_{i-l} \otimes p_{n-i}.
\end{align*}
As the last sum is equivalent to the right hand side of Eq.~\eqref{eq:MUCoprod} up to reindexing of the sum, we are done.
\end{proof}

Our next goal is to understand the product of $\MU_*(\CP^\infty)$, which is given by pushforward along the infinite Segre embedding $\sigma \colon \CP^\infty \times \CP^\infty \to \CP^\infty$. Note that $\MU^*(\CP^\infty \times \CP^\infty) \cong \MU^*[[x_1, x_2]]$, where $x_i$ is the pullback of the first Chern class of the tautological line bundle $\Oc(1)$ on the $i$th factor. As $s^*(\Oc(1)) = \Oc(1,1)$, the pullback of $x$ along the Segre embedding can be expressed in terms of the formal group law \cite[Lemma~3.3(f)]{ravenel-wilson}
\begin{equation}
    \sigma^*(x) = \sum_{i,j \geq 0} a_{ij} x_1^i \smile x_2^j \in  \MU^*(\CP^\infty \times \CP^\infty).
\end{equation}
From this equation, we can derive a formula for $s_*$ in terms of the dual basis $\beta_i$.

\begin{lem}\label{lem:DualBasisMUProd}
The equality
\[
\sigma_*(\beta_i \otimes \beta_j) = \sum_{k = 0}^{i+j} a^{(k)}_{ij} \beta_k
\]
holds, where $a^{(k)}_{nm}$ is the coefficient of $x_1^n x_2^m$ in $(\sum_{i,j \geq 0} a_{ij}x_1^i x_2^j)^k$.
\end{lem}
\begin{proof}
The coefficient of $\beta_k$ in $\sigma_*(\beta_i \otimes \beta_j)$ coincides with the pushforward of 
\[
x^k \frown \sigma_*(\beta_i \otimes \beta_j) \in \MU_{2(i+j-k)}(\CP^\infty)
\]
to $\MU_{2(i+j-k)}$. By projection formula, this coincides with the pushforward of 
\[
\bigg( \sum_{i,j \geq 0} a_{ij} x_1^i \smile x_2^j \bigg)^k \frown \beta_i \otimes \beta_j \in \MU_{2(i+j-k)}(\CP^\infty \times \CP^\infty)
\]
to $\MU_{2(i+j-k)}$, which is exactly $a^{(k)}_{ij}$.
\end{proof}

Combining the above results, we can understand the product in terms of the $p_i$ basis. 

\begin{prop}\label{prop:MUProd}
We have
\[
\sigma_*(p_n \otimes p_m) = \sum_{r=0}^{n+m} s^{(r)}_{n,m} \cdot p_r \in \MU_{2(n+m)}(\CP^\infty),
\]
where 
\[
s^{(r)}_{n,m} := \sum_{i,j,k \geq 0} [\CP^i] \cdot [\CP^j] \cdot a^{(r+k)}_{n-i,m-j} \cdot a_{1,k} \in \MU_{2(n+m-r)}. 
\]
\end{prop}
\begin{proof}
By Eq.~\ref{eq:MUBaseChange2}, we write 
\begin{align*}
\sigma_*(p_n \otimes p_m) &= \sigma_*((\sum_i^n [\CP^i]\cdot \beta_{n-i}) \otimes (\sum_j^m [\CP^j]\cdot \beta_{m-j})) \\
&= \sum_{i,j} [\CP^i] \cdot [\CP^j] \cdot \sigma_*(\beta_{n-i}\otimes \beta_{m-j}).
\end{align*}
By Lemma~\ref{lem:DualBasisMUProd}, this equation expands out to
\[
    \sum_{i,j} [\CP^i]\cdot[\CP^j] \cdot \sigma_*(\beta_{n-i}\otimes \beta_{m-j}) = \sum_{i,j} [\CP^i]\cdot[\CP^j] \cdot (\sum_{k=0}^{n-i+m-j} a^{(k)}_{n-i,m-j} \cdot \beta_k)
\]
and by Eq.~\eqref{eq:MUBaseChange1}, the right hand term is 
\[
\sum_{i,j}[\CP^i]\cdot[\CP^j] \cdot  (\sum_{k=0}^{n-i+m-j}  a^{(k)}_{n-i,m-j} \cdot (\sum_{l=0}^k a_{1,l}\cdot p_{k-l}))).
\]
Rearranging the sum, we note that the index for $p_r$ ranges from $0$ to $n+m$.
\end{proof}

\begin{rem}
Over a field $k$ of characteristic 0, Lee and Pandharipande consider algebraic cobordism classes $[X,\Ec]$ of smooth varieties equipped with a vector bundle \cite{pandharipande:2012}. In the rank-one case, such classes form the \textit{algebraic cobordism ring of line bundles} $\Omega_{*,1}(k)$, which is an algebra over the algebraic bordism ring $\Omega_*(k)$ of Levine--Morel. The product is given by formula 
\[
[X_1, \Ls_1] \cdot [X_2, \Ls_2] = [X_1 \times X_2, \Ls_1 \boxtimes \Ls_2].
\]
They prove that $\Omega_{*,1}(k)$ is a free $\Omega_*(k)$-module with basis given by
\[
\tilde p_i := [\Pb^i, \Oc(1)] \in \Omega_{i,1}(k).
\]
Lee and Pandharipande then ask \cite[Section~3]{pandharipande:2012} for an explicit formula for the product $\Omega_{i,1}(k)$ in terms of this basis.

There is a natural map $\eta \colon \Omega_*(\Pb^\infty) \to \Omega_{*,1}(k)$ given by
\[
[X \stackrel{f}{\to} \Pb^\infty] \mapsto [X; f^* \Oc(1)],
\]
which is an isomorphism as $\Omega_*$ satisfies projective bundle formula \cite[Section~3.5]{levine-morel}, and therefore $\Omega_*(\Pb^\infty)$ is a free $\Omega_*(k)$-module with a basis given by the lifts of $\tilde p_i$.\footnote{More precisely, $\Omega_*(\Pb^\infty) := \colim_{i} \Omega_*(\Pb^i)$, where the colimit is taken along a fixed sequence of linear embeddings $\Pb^0 \hook \Pb^1 \hook \cdots \hook \Pb^i \hook \Pb^{i+1} \hook \cdots$.} The map $\eta$ is an isomorphism of rings, where the ring structure on $\Omega_*(\Pb^\infty)$ is given by the pushforward along the Segre embedding. Moreover, as $\Omega_*(k) \cong \MU_{2*}$ \cite[Section~4.3]{levine-morel}, we see that there is a natural isomorphism of Hopf algebras
\[
\MU_{2*}(\CP^\infty) \cong \Omega_{*}(\Pb^\infty),
\]
and therefore the product structure of $\Omega_{*,1}(k)$ is completely determined by that of $\MU_{2*}(\CP^\infty)$.
\end{rem}

As a consequence of the above discussion, we obtain the following result.

\begin{cor}\label{cor:LeePQuest}
Let $k$ be a field of characteristic 0. Then 
\[
[\Pb^n \times \Pb^m; \Oc(1,1)] = \sum_{r=0}^{n+m} \tilde s^{(r)}_{n,m} \cdot  [\Pb^r, \Oc(1)] \in \Omega_{m+n,1}(k)
\]
where 
\[
\tilde s^{(r)}_{n,m} := \sum_{i,j,k \geq 0} [\Pb^i] \cdot [\Pb^j] \cdot a^{(r+k)}_{n-i,m-j} \cdot a_{1,k} \in \Omega_{n+m-r}(k). 
\]
\end{cor}

\section{Category \rev{of} motivic spectra satisfying derived blowup excision: $\MS_S^\dbe$}\label{sect:MSdbe}

Here, we construct a version of motivic spectra that represent cohomology theories satisfying excision in derived blowups in the sense of Khan--Rydh \cite{khan-rydh}, and compare it with some other natural variants of motivic spectra. Our exposition follows closely that of \cite[Section~2]{annala-iwasa:MotSp} and \cite{AHI}. Throughout this section, $S$ is a derived scheme, and $\Sch^\afp_S$ is the $\infty$-category of derived schemes almost of finite presentation\footnote{\rev{See e.g. \cite[Definition~68]{annala-thesis} for the definition of almost finite presentation.}} over $S$.

If $\Cc$ is a category, we will denote by $\Pc(\Cc, \Dc)$ the category of $\Dc$-valued presheaves on $\Cc$. The notation $\Pc(\Cc)$ stands for presheaves valued in spaces, and together with its cartesian symmetric monoidal structure, it is the prototypical example of a presentably symmetric monoidal category. Stabilization $\Sigma^\infty_+ \colon \Pc(\Cc) \to \Pc(\Cc, \Sp)$ is an example of a map of presentably symmetric monoidal categories when the target is equipped with the pointwise smash product.

By decorating $\Pc$ with subscripts, we mean the full subcategory of $\Pc(\Cc, \Dc)$ satisfying some conditions, such as descent with respect to a Grothendieck topology. This is best illuminated by an example.

\begin{defn}
Denote by $\Pc_{\Nis, \dbe}(\Sch^\afp_S, \Sp)$ the category of Nisnevich sheaves of spectra on $\Sch^\afp_S$ that satisfy \emph{derived blowup excision}, i.e., send squares
\begin{equation}\label{eq:dbe}
\begin{tikzcd}
    E \arrow[r,hook] \arrow[d] & \Bl_Z(X) \arrow[d] \\
    Z \arrow[r,hook]{}{i} & X,
\end{tikzcd}    
\end{equation}
where $i$ is a derived regular embedding, and $E$ is the exceptional divisor of the blowup, to cartesian square of spectra.
\end{defn}

The fully faithful inclusion $\Pc_{\Nis, \dbe}(\Sch^\afp_S, \Sp) \hook \Pc(\Sch^\afp_S, \Sp)$ to presheaves of spectra is the right adjoint of a symmetric monoidal accessible localization\footnote{It suffices to check that the symmetric monoidal structure on $\Pc(\Sch^\afp_S, \Sp)$ preserves Nis and dbe equivalences separately in both variables (see e.g. \cite[Proposition~A.5]{nikolaus-scholze}). This follows from the fact that both Nisnevich and derived blowup squares are preserved under pullbacks along maps of derived schemes. All the localizations considered in this section are symmetric monoidal for similar reason.} (see \cite[Section~5.5.4]{HTT})
\[
L_{\Nis,\dbe} \colon \Pc(\Sch^\afp_S, \Sp) \to \Pc_{\Nis, \dbe}(\Sch^\afp_S, \Sp).
\]
Thus, one can associate functorially to each $X \in \Sch^\afp_S$ the object
\[
X_+ := L_{\Nis,\dbe}(\underline{X}) \in \Pc_{\Nis, \dbe}(\Sch^\afp_S, \Sp),
\]
where $\underline X$ is the presheaf of spectra given by suspension spectrum of the functor of points of $X$. By slight abuse of notation, we denote by
\[
X := \cofib(S_+ \to X_+) \in \Pc_{\Nis, \dbe}(\Sch^\afp_S, \Sp)
\]
the object associated to a pointed object $S \to X$ in $\Sch^\afp_S$.

Recall that the symmetric monoidal structure on $\Pc_{\Nis, \dbe}(\Sch^\afp, \Sp)$ is uniquely determined by the symmetric monoidality of 
\[
\Sch^\afp_S \to \Pc_{\Nis, \dbe}(\Sch^\afp_S, \Sp),
\]
i.e., $X_+ \otimes Y_+ \simeq (X \times Y)_+$, together with the fact that $\otimes$ commutes with colimits separately in both variables. 

Recall that in topology, the category of spectra is obtained from the category of pointed spaces by making the suspension endomorphism an automorphism in a universal fashion. As tensoring with the projective line $\Pb^1$, pointed at $\infty$, is an algebro-geometric version of suspension, we define motivic spectra by inverting the functor $\Pb^1 \otimes -$ using the general notion of $\otimes$-inversion (see e.g. \cite[Section~1]{annala-iwasa:MotSp}).

\begin{defn}[Motivic spectra satisfying derived blowup excision]\label{def:MSdbe}
We define
\[
\MS^\dbe_S := \Pc_{\Nis, \dbe}(\Sch^\afp_S, \Sp)[(\Pb^1)^{-1}].
\]
\end{defn}

For the convenience of the reader, we recall that the usual stable $\Ab^1$-invariant motivic homotopy category is defined as
\[
\SH_S := \Pc_{\Nis, \Ab^1}(\Sm_S, \Sp)[(\Pb^1)^{-1}]
\]
and the category of motivic spectra considered in \cite{AHI:atiyah} is defined as 
\[
\MS_S := \Pc_{\Nis, \sbe}(\Sm_S, \Sp)[(\Pb^1)^{-1}],
\]
where $\Ab^1$ stands for $\Ab^1$-invariance, and sbe stands for \emph{smooth blowup excision}.\footnote{A presheaf on $\Sm_S$ satisfies smooth blowup excision if it sends all smooth blowup squares (\rev{i.e.} squares like in Eq.~\eqref{eq:dbe} but where $X,Z \in \Sm_S$) to Cartesian squares.} There is a symmetric monoidal localization $\MS_S \to \SH_S$ called $\Ab^1$-localization.

\begin{prop}
The left Kan extension functor $\Pc(\Sm_S) \to \Pc(\Sch^\afp_S)$ induces a symmetric monoidal functor
\[
\MS_S \to \MS^\dbe_S.
\]
\end{prop}
\begin{proof}
As every smooth blowup square is a derived blowup square, left Kan extension induces a symmetric monoidal colimit preserving functor
\[
\Pc_{\Nis, \sbe}(\Sm_S, \Sp) \to \Pc_{\Nis, \dbe}(\Sch^\afp_S, \Sp).
\]
The claim then follows from the universal property of $\otimes$-inversion, see e.g. \cite[Proposition~2.9]{robalo:2015}.
\end{proof}

It is an interesting question whether or not the above functor is fully faithful. An analogous claim in $\Ab^1$-homotopy theory was investigated in \cite{khan:cdh}.

Next, we define versions of stable motivic homotopy theory that satisfy excision for abstract blowup squares. 

\begin{defn}[Abe-local motivic spectra]
We define \emph{abe-local motivic spectra} as
\[
\MS^\abe_S := \Pc_{\Nis,\abe}(\Sch^\afp_S, \Sp)[(\Pb^1)^{-1}],
\]
and \emph{abe-local $\Ab^1$-invariant motivic spectra} as
\[
\SH^\abe_S := \Pc_{\Nis,\abe, \Ab^1}(\Sch^\afp_S, \Sp)[(\Pb^1)^{-1}],
\]
where abe stands for \emph{abstract blowup excision}, and $\Ab^1$ stands for $\Ab^1$-invariance.
\end{defn}

As derived blowup squares are abstract blowup squares,\footnote{\rev{Indeed, the structure morphism of a derived blowup is a proper morphism that is an equivalence away from the center, see \cite[Theorem~4.1.5]{khan-rydh}}.} we obtain the following result.

\begin{prop}
There exists a sequence
\[
\MS^\dbe_S \to \MS^\abe_S \to \SH^\abe_S
\]
of symmetric monoidal localizations.
\end{prop}
\begin{proof}
Indeed, the symmetric monoidal localizations
\[
L_{\abe} \colon \Pc_{\Nis, \dbe}(\Sch^\afp_S, \Sp) \to \Pc_{\Nis,\abe}(\Sch^\afp_S, \Sp)
\]
and
\[
L_{\abe,\Ab^1} \colon \Pc_{\Nis, \dbe}(\Sch^\afp_S, \Sp) \to \Pc_{\Nis,\abe, \Ab^1}(\Sch^\afp_S, \Sp)
\]
induce localizations after inverting $\Pb^1$ by \cite[Lemma~1.5.4]{annala-iwasa:MotSp}.
\end{proof}

Above, we have defined $\SH^\abe_S$ starting from the the category of almost finitely presented derived $S$-schemes. Next, we compare it to the cdh-local version of stable $\Ab^1$-homotopy theory (see e.g. \cite{voevodsky:Hcdh}). We prove that, at least over noetherian derived schemes\footnote{Recall that a derived scheme $S$ is \emph{noetherian} if its underlying classical scheme $S_\cl$ is noetherian, and if the homotopy sheaves $\pi_i(\Oc_S)$ are coherent sheaves on $S_\cl$.} abe-local $\Ab^1$-invariant motivic spectra over $S$ coincide with the cdh-local stable $\Ab^1$-homotopy category over the classical truncation $S_\cl$.

The reason for restricting to the noetherian case is because the condition of being almost of finite presentation simplifies in that setting: if $S$ is noetherian, then a derived $S$-scheme $X$ is almost of finite presentation if and only if $X$ is noetherian and the map of underlying classical schemes $X_\cl \to S_\cl$ is of finite type. In particular, the canonical inclusion $S_\cl \hook S$ is almost of finite presentation.

\begin{prop}\label{prop:SHcdhNilInv}
Let $S$ be a noetherian derived scheme. Then there is a canonical equivalence
\[
\underline{\SH}^\cdh_{S_\cl} \cong \SH^\abe_S,
\]
where $S_\cl$ denotes the underlying classical scheme of $S$, and where
\[
\underline{\SH}^\cdh_{S_\cl} := \Pc_{\cdh, \Ab^1}(\Sch^\fp_{\tau(S)})_*[(\Pb^1)^{-1}]
\]
is defined as in \cite{khan:cdh}.
\end{prop}

Above, for a qcqs classical scheme $X$, $\Pc_{\cdh, \Ab^1}(\Sch^\fp_{X})_*$ denotes the category of $\Ab^1$-invariant cdh-sheaves of pointed spaces on finitely presented classical $X$-schemes.

\begin{proof}
Let $X \in \Sch^\afp_S$. As the square
\[
\begin{tikzcd}
    \emptyset \arrow[d] \arrow[r] & \emptyset \arrow[d] \\
    X_\cl \arrow[r,hook]& X
\end{tikzcd}
\]
is an abstract blowup square, presheaves on $\Sch^\afp_S$ satisfying abstract blowup excision are \emph{nil-invariant} in the sense that they induce equivalences when pulling back along the inclusion $X_\cl \hook X$ of the underlying classical scheme. Thus, the inclusion $\Sch^\fp_{S_\cl} \hook \Sch^\afp_{S}$, and the truncation $\Sch^\afp_{S} \to \Sch^\fp_{S_\cl}$ functors induce an equivalence of categories
\[
\Pc_{\Zar,\nil}(\Sch^\afp_S)_* \simeq \Pc_{\Zar,\nil}(\Sch^\fp_{S_\cl})_*
\]
of nil-invariant Zariski sheaves.

By a further localization, we obtain a natural equivalence
\[
\Pc_{\Nis,\abe,\Ab^1}(\Sch^\afp_S)_* \simeq \Pc_{\cdh,\Ab^1}(\Sch^\fp_{S_\cl})_*.
\]
Above, we have used the fact that a Nisnevich sheaf on a category of classical schemes, such as $\Sch^\fp_{S_\cl})_*$, is cdh-sheaf if and only if it satisfies excision for abstract blowup squares \cite[Theorem~2.2]{voevodsky:Hcdh}. Furthermore, by $\otimes$-inverting $\Pb^1$, we obtain an equivalence
\[
\Pc_{\Nis,\abe,\Ab^1}(\Sch^\afp_S)_*[(\Pb^1)^{-1}] \simeq \underline{\SH}^\cdh_{S_\cl}.
\]
It remains to be verified that the left hand side is $\SH^\abe_S$. Note that $\Ab^1$-invariance implies that $\Pb^1 \simeq S^1 \otimes \Gb_m$. As
\[
\Pc_{\Nis,\abe,\Ab^1}(\Sch^\afp_S)_*[(S^1)^{-1}] = \Pc_{\Nis,\abe,\Ab^1}(\Sch^\afp_S, \Sp)
\]
and as inverting $\Pb^1$ is the same thing as first inverting $S^1$ and then inverting $\Pb^1$, we have proven the claim.
\end{proof}

To summarize, we have shown the following.

\begin{prop}\label{prop:DiagOfHtyThies}
Let $S$ be a noetherian derived scheme. Then there is a diagram
\begin{equation}
\begin{tikzcd}
    \SH^\abe_S  \arrow[r,hook] & \MS^\abe_S \arrow[r,hook] & \MS^\dbe_S \\
    \SH_S \arrow[rr,hook] \arrow[u,hook] & & \MS_S, \arrow[u]
\end{tikzcd}    
\end{equation}
where the hooked arrows stand for fully faithful functors, where the horizontal arrows are right adjoints of symmetric monoidal localizations, and where the vertical arrows are symmetric monoidal.
\end{prop}
\begin{proof}
The only part that requires proof is the fully faithfulness of the left vertical map. The analogous claim is known for the stable motivic homotopy categories of classical schemes \cite{khan:cdh}. Thus the claim follows from Proposition~\ref{prop:SHcdhNilInv}, combined with the nil-invariance result for $\SH_S$ from Adeel Khan's thesis \cite{khan:thesis}.
\end{proof}

\begin{rem}[Cdh-topology vs abstract blowup excision on derived schemes]
In the original version of this paper, we tried to formulate the results of this section using a version of $\MS_S$ (and $\SH_S$) built from ($\Ab^1$-invariant) cdh-sheaves on $\Sch^\afp_S$ instead of Nisnevich sheaves that satisfy abstract blowup excision. However, Marc Hoyois pointed out to us that a cdh-sheaf on derived schemes need not satisfy excision in abstract blowup squares, as abstract blowup squares do not form a regular cd-structure in the sense of \cite[Definition~2.10]{voevodsky:cd} on derived schemes. This is because closed embeddings are not monomorphisms of derived schemes.
\end{rem}

\section{Representability in $\MS_S^\dbe$}\label{sect:Hrep}

Here, we prove that the Elmanto--Morrow motivic cohomology is representable in $\MS_k^\dbe$ using a slight variant of \cite[Construction~6.5]{AHI:atiyah}. Based on the explicit construction of $\Pb^1$-spectra provided in \cite[Section~1]{annala-iwasa:MotSp}, the representability amounts to providing motivic cohomology with the structure of a symmetric spectrum, i.e., 
\begin{enumerate}
    \item constructing \textit{$\Pb^1$-deloopings} $\sigma \colon \Zb(i)^\mot \stackrel{\sim}{\to} \left( \Zb(i+1)^\mot \right)^{\Pb^1}$;
    \item such that the compositions of the adjoints of $\sigma$ induce $\Sigma_i$-equivariant maps 
    \[
        \Pb^1 \otimes \dots \otimes \Pb^1 \otimes \Zb(j)^\mot \to \Zb(i+j)^\mot,
    \]
    where the action of the symmetric group on $\Zb(k)^\mot$ is trivial.
\end{enumerate}
The deloopings are provided by the first Chern class of $\Oc(1)$ on the projective line. The second point is automatically taken care of by the fact that the motivic complexes form an $\Eb_\infty$-algebra of sheaves of graded spectra.

We begin by providing a general description of $\Eb_\infty$-algebra objects in the category $\Sp_c(\Cc)$ of $c$-spectra, where $\Cc$ is a presentably symmetric monoidal $\infty$-category. We will use the notation of \cite[Section~1]{annala-iwasa:MotSp}, except that we denote the category of \emph{symmetric sequences} by $\SSeq(\Cc) := \Fun(\Fin^\simeq, \Cc)$ instead of $\Cc^\Sigma$. Moreover, we denote by $\Gr_\Nb(\Cc) := \Fun(\Nb, \Cc)$ the category of $(\Nb\dash)$graded objects in $\Cc$, where $\Nb$ is the discrete category whose objects are the natural numbers.\footnote{For us, 0 is a natural number.} There is a symmetric monoidal functor $\Fin^\simeq \to \Nb$ where the symmetric monoidal structures are disjoint union on the source and addition on the target. Consequently, precomposition induces a lax symmetric monoidal functor
\begin{equation}\label{eq:FromGrToSSeq}
\Gr_\Nb(\Cc) \to \SSeq(\Cc)   
\end{equation}
where the symmetric monoidal structures on the above categories are those given by the Day convolution. Thus, a graded $\Eb_\infty$-algebra in $\Cc$ can be canonically regarded as an $\Eb_\infty$-algebra in $\SSeq(\Cc)$.

\begin{lem}\label{lem:EooInLaxSp}
An $\Eb_\infty$-algebra in $\Sp_c^\lax(\Cc)$ is equivalent to an $\Eb_\infty$-algebra $E \in \SSeq(\Cc)$ together with a map $\sigma \colon c \to E_1$ in $\Cc$.
\end{lem}
\begin{proof}
Indeed, the category of lax $c$-spectra in $\Cc$ is by definition the category of $S_c$-modules in $\SSeq(\Cc)$, where $S_c$ is the free $\Eb_\infty$-algebra in $\SSeq(\Cc)$ generated by the symmetric sequence obtained by placing $c$ in degree one, and the initial object everywhere else. As an $\Eb_\infty$-algebra $E$ in $S_c$-modules is the same thing as a map of $\Eb_\infty$-algebras $S_c \to E$ in $\SSeq(\Cc)$, the claim follows from the universal property of free $\Eb_\infty$-algebras. 
\end{proof}

\begin{lem}\label{lem:LaxSpToSp}
Let everything be as in Lemma~\ref{lem:EooInLaxSp}. Then an $\Eb_\infty$-algebra in $\Sp^\lax_c(\Cc)$ lies in the full subcategory $\Sp_c(\Cc)$ of $c$-spectra if and only if the multiplication-by-$\sigma$-maps $\sigma_i \colon c \otimes E_i \to E_{i+1}$ induce adjoint equivalences $E_i \stackrel \sim \to (E_{i+1})^c$.
\end{lem}
\begin{proof}
Indeed, the multiplication by $\sigma$ induces a map of lax $c$-spectra
\[
\sigma \colon s_+ (c \otimes E) \to E
\]
and the condition of being a $c$-spectrum is exactly that the induced adjoint map
\[
\sigma^\sharp \colon E \to s_-E^c
\]
is an equivalence \cite[Definition~1.3.8]{annala-iwasa:MotSp}. Since the forgetful functor $\Sp^\lax_c(\Cc) \to \SSeq(\Cc)$ is conservative, it suffices to check this on the underlying symmetric sequences. By unwinding the definitions, we see that this is exactly the condition in the statement.
\end{proof}

The above general discussion allows us to provide convenient criterion for representability of cohomology theories in motivic spectra. For this purpose, \rev{let $S$ be a derived scheme and} let $\Hc$ be a symmetric monoidal accessible localization of $\Pc(\Sch^\afp_S; \Sp)$. 

\begin{defn}
\begin{enumerate}
    \item A \emph{pre-orientation} on an $\Eb_\infty$-algebra object $E \in \Gr_\Nb(\Hc)$ is a point
    \[
    c_1 \in E_1(\Ps ic_S, S),\footnote{\rev{The \emph{relative cohomology} $E(Y,X)$ associated to a morphism $f \colon X \to Y$ is the fibre of $f^* \colon E(Y) \to E(X)$}.}
    \]
    where the Picard stack over $S$, $\Ps ic_S$, is regarded as a pointed object with respect to the map $S \to \Ps ic_S$ classifying the structure sheaf of $S$.
    
    \item If $X \in \Sch^\afp_S$ and $\Ls$ is a line bundle on $X$ that is classified by the map $f_\Ls \colon X \to \Ps ic_S$, then the \emph{(first) Chern class} of $\Ls$ is 
    \[
    c_1(\Ls) := f_\Ls^*(c_1) \in E_1(X).
    \]
    Notice that, by definition, a trivialization of a line bundle provides a trivialization of its first Chern class.
    
    \item A pre-orientation $c_1$ on $E$ is an \emph{orientation} if $E$ satisfies the projective bundle formula with respect to $c_1$, i.e., if the map given by
    \[
    \bigoplus_{i=0}^n c_1(\Oc(1))^i \rev{\pi^*} \colon \bigoplus_{i=0}^n E_{m-i}(X) \stackrel{\sim}{\to} E_{m}(\Pb^n_X),
    \]
    is an equivalence for all $X$ and $n,m \geq 0$\rev{, where $\pi \colon \Pb^n_X \to X$ is the natural structure map,} and where by convention $E_i = 0$ for $i < 0$.
\end{enumerate}
\end{defn}

\rev{Suppose $E$ is an $\Eb_\infty$-algebra object in $\Gr_\Nb(\Hc)$, and $c_1 \in E(\Ps ic_S, S)$ is a pre-orientation on $E$.} As the tautological bundle $\Oc(1)$ on $\Pb^1_S$ restricts to a trivial line bundle along $\infty \colon S \to \Pb^1_S$, we obtain a lift
\begin{equation}\label{eq:ChLift}
\sigma \in E_1(\Pb^1_S, S)    
\end{equation}
of the first Chern class to the relative cohomology.\footnote{One might worry about the dependence of the lift on the choice of the trivialization $\gamma \colon \Oc(1) \vert_S \simeq \Oc_S$. Fortunately, the lift turns out to be independent of this choice. Indeed, if $\gamma'$ is another trivialization, then $\gamma$ and $\gamma'$ differ by an automorphism $\alpha$ of $\Oc_S$, i.e., an invertible function on $S$. As $\alpha$ lifts to a (unique) automorphism $\tilde \alpha$ of $\Oc(1)$, we obtain a path in $E(\Pb^1,S)$ connecting the two lifts of $c_1(\Oc(1))$.}

\begin{thm}\label{thm:MotSpRep}
If $E \in \Gr_\Nb(\Hc)$ is an oriented $\Eb_\infty$-algebra, then the class $\sigma$ defined in Eq.~\eqref{eq:ChLift} endows the symmetric sequence associated to $E$ (via Eq.~\eqref{eq:FromGrToSSeq}) with the structure of an oriented $\Eb_\infty$-algebra in $\Sp_{\Pb^1_S}(\Hc)$.
\end{thm}
\begin{proof}
Indeed, multiplication by $\sigma$ induces maps
\[
\sigma \colon E_i(X) \stackrel{\sim}{\to} E_{i+1}(\Pb^1_X,X),
\]
which are equivalences by the $n=1$ case of the projective bundle formula. In other words, $\sigma$ induces equivalences $E_i \stackrel \sim \to (E_{i+1})^{\Pb^1_S}$, where in the last formula $\Pb^1_S$ is regarded as a pointed object with at basepoint at infinity. Thus the claim follows from Lemma~\ref{lem:LaxSpToSp}.
\end{proof}

\rev{
Let $k$ be a field. As the motivic sheaves on almost finite type $k$-schemes satisfy Nisnevich descent and derived blowup excision, they form objects $\Zb^\mot(j) \in \Pc_{\Nis,\dbe}(\Sch_k^\afp; \Sp)$. The $\Nb$-graded multiplicative structure of motivic cohomology assembles these sheaves into an $\Eb_\infty$-algebra object 
\[
\Zb(*)[2*] \in \Gr_\Nb(\Pc_{\Nis,\dbe}(\Sch_k^\afp; \Sp)).\footnote{\rev{Note that we have shifted the $j$th motivic complex by $2j$ topological degrees, as it is these shifted motivic complexes that occur as the graded pieces on the motivic filtration on algebraic $K$-theory \cite[Theorem~1.1]{elmanto-morrow}. The $\Nb$-graded $\Eb_\infty$-algebra structure on motivic cohomology is then obtained by passing to graded pieces, which is symmetric monoidal \cite[\S 2.2]{elmanto-morrow}.}}
\]
Finally, as the map of \cite[Lemma~5.1]{elmanto-morrow} provides a pre-orientation for motivic cohomology, with respect to which it satisfies the projective bundle formula \cite[Theorem~5.3]{elmanto-morrow}, we may apply Theorem~\ref{thm:MotSpRep} to deduce the following representability result for motivic cohomology.
}


\begin{cor}
Let $k$ be a field. Then the symmetric sequence associated to the $\Nb$-graded sheaf of spectra whose graded pieces are $\Zb(j)^\mot[2j] \colon \Sch^{\afp,\op}_k \to \Sp$ lifts to an object in $\MS^\dbe_k$.
\end{cor}

\rev{
\begin{rem}
Similarly, as mixed characteristic motivic cohomology satisfies Nisnevich descent, smooth blowup excision, and projective bundle formula \cite{bouis1,bouis2}, we deduce that for every qsqs scheme $S$, motivic cohomology is representable in $\MS_S$. If we knew that mixed characteristic motivic cohomology satisfied derived blowup excision, then we could upgrade this representability result into representability in $\MS^\dbe_S$.
\end{rem}
}
\bibliographystyle{alphamod}
\bibliography{references}

\end{document}